\documentclass[12pt,fleqn,draft]{article} 
\usepackage{amsmath,amssymb,amsthm,amsfonts,bm}
\usepackage{enumerate}
\usepackage{indentfirst}
\usepackage{cite}
\usepackage[usenames,dvipsnames]{color}

\makeatletter
\def\@cite#1#2{[{{\bfseries #1}\if@tempswa , #2\fi}]}
\renewcommand{\section}{%
\@startsection{section}{1}{\z@}
{0.5truecm plus -1ex minus -.2ex}%
{1.0ex plus .2ex}{\bfseries\large}}
\def\@seccntformat#1{\csname the#1\endcsname.\ }
\makeatother

\numberwithin{equation}{section} 
\theoremstyle{theorem}
\newtheorem{thm}{Theorem}[section]

\newtheorem{lem}[thm]{Lemma}

\theoremstyle{definition}
\newtheorem{df}{Definition}[section]
\newtheorem{remark}{Remark}[section]
\newtheorem{example}{Example}[section]

\newtheorem*{prth1.1}{Proof of Theorem 1.1}
\newtheorem*{prth1.2e}{Proof of Theorem 1.2 (existence part)}
\newtheorem*{prth1.2u}{Proof of Theorem 1.2 (uniqueness part)}
\newtheorem*{prth1.3}{Proof of Theorem 1.3}

\newcommand{\ep}{\varepsilon}

\let\hat\widehat

\def\Pi{\hat\pi}

\begin{document}
\begin{center}
    \Large{{\bf Time discretization of 
an abstract problem \\ 
applying to the linearized equations \\ 
of coupled sound and heat flow}}
\end{center}
\vspace{5pt}
\begin{center}
    Shunsuke Kurima%
    \\
    \vspace{2pt}
    Department of Mathematics, 
    Tokyo University of Science\\
    1-3, Kagurazaka, Shinjuku-ku, Tokyo 162-8601, Japan\\
    {\tt shunsuke.kurima@gmail.com}\\
    \vspace{2pt}
\end{center}
\begin{center}    
    \small \today
\end{center}

\vspace{2pt}
\newenvironment{summary}
{\vspace{.5\baselineskip}\begin{list}{}{%
     \setlength{\baselineskip}{0.85\baselineskip}
     \setlength{\topsep}{0pt}
     \setlength{\leftmargin}{12mm}
     \setlength{\rightmargin}{12mm}
     \setlength{\listparindent}{0mm}
     \setlength{\itemindent}{\listparindent}
     \setlength{\parsep}{0pt}
     \item\relax}}{\end{list}\vspace{.5\baselineskip}}
\begin{summary}
{\footnotesize {\bf Abstract.} 
In this paper 
we deal with an abstract problem 
which includes 
the linearized equations of coupled sound and heat flow 
as an example. 
Recently, 
a time discretization of a simultaneous abstract evolution equation 
applying to some parabolic-hyperbolic phase-field systems 
has been studied. 
This paper focuses on a time discretization of 
an abstract problem applying to the linearized equations of coupled sound and heat flow. 
Also, this paper gives  
some parabolic-hyperbolic phase-field systems 
as examples.  
}
\end{summary}
\vspace{10pt}

\newpage

\section{Introduction} \label{Sec1}

Matsubara--Yokota \cite{MY2016} 
have established existence and uniqueness of solutions to 
the initial-boundary value problem for the linearized equations 
of coupled sound and heat flow  
\begin{equation*}
     \begin{cases}
         \theta_{t} + (\gamma-1)\varphi_{t} - \sigma\Delta\theta = 0      
         & \mbox{in}\ \Omega\times(0, \infty), 
 \\[1mm]
        \varphi_{tt}  - c^2 \Delta\varphi 
        -m^2 \varphi = -c^2 \Delta\theta 
         & \mbox{in}\ \Omega\times(0, \infty), 
 \\[1mm] 
         \theta = \varphi = 0 
         & \mbox{on}\ \partial\Omega\times(0, \infty),  
 \\[1mm]
         \theta(0) = \theta_{0},\ \varphi(0) = \varphi_{0},\ \varphi_{t}(0) = v_{0}  
         &\mbox{in}\ \Omega                                      
     \end{cases}
\end{equation*}
by applying the Hille--Yosida theorem 
and have derived regularity of solutions, 
where $c>0$, $\sigma>0$, $m \in \mathbb{R}$ and $\gamma>1$ are constants, 
$\Omega \subset \mathbb{R}^{d}$ ($d \in \mathbb{N}$) 
is a domain with smooth bounded boundary $\partial\Omega$
and $\theta_{0}$, $\varphi_{0}$, $v_{0}$ are given functions.   

The paper \cite{K4} 
has proved   
existence of solutions to 
the initial valued problem for the simultaneous abstract evolution equation 
\begin{equation*}
     \begin{cases}
         \dfrac{d\theta}{dt} + \dfrac{d\varphi}{dt} + A_{1}\theta = f      
         & \mbox{in}\ (0, T), 
 \\[4mm]
 L\dfrac{d^2\varphi}{dt^2} + B\dfrac{d\varphi}{dt} + A_{2}\varphi
         + \Phi\varphi + {\cal L}\varphi = \theta   
         & \mbox{in}\ (0, T), 
 \\[2mm] 
       \theta(0) = \theta_{0},\ \varphi(0) = \varphi_{0},\ \dfrac{d\varphi}{dt}(0) = v_{0}      
     \end{cases}
 \end{equation*}
by employing a time discretization scheme 
in reference to \cite{CF1996, CK} 
and 
has obtained an error estimate for the difference 
between continuous and discrete solutions.  
Here $T>0$,  
$L : H \to H$ is a linear positive selfadjoint operator, 
$B : D(B) \subset H \to H$, $A_{j} : D(A_{j}) \subset H \to H$ ($j = 1, 2$) are 
linear maximal monotone selfadjoint  operators, 
$H$ and $V$ are real Hilbert spaces satisfying $V \subset H$, 
$V_{j}$ ($j=1, 2$) are linear subspaces of $V$ 
satisfying $D(A_{j}) \subset V_{j}$ ($j=1, 2$), 
$\Phi : D(\Phi) \subset H \to H$ is a maximal monotone operator,   
${\cal L} : H \to H$ is a Lipschitz continuous operator,   
$f : (0, T) \to H$ and $\theta_{0} \in V_{1}$, $\varphi_{0}, v_{0} \in V_{2}$ 
are given. 
Moreover, 
the paper \cite{K4} has assumed some conditions 
in reference to \cite[Section 2]{CF1996} 
and assumptions in 
\cite{CK, GP2003, GP2004, GPS2006, WGZ2007, WGZ2007dynamicalBD}, 
and has given    
some parabolic-hyperbolic phase-field systems 
under 
homogeneous Dirichlet--Dirichlet boundary conditions 
or homogeneous Dirichlet--Neumann boundary conditions 
or homogeneous Neumann--Dirichlet boundary conditions 
or homogeneous Neumann--Neumann 
boundary conditions as examples.  

In this paper we consider 
the abstract problem 
%
%
 \begin{equation*}\tag*{(P)}\label{P}
     \begin{cases}
         \dfrac{d\theta}{dt} + \eta\dfrac{d\varphi}{dt} + A_{1}\theta = 0     
         & \mbox{in}\ (0, T), 
 \\[4mm]
 L\dfrac{d^2\varphi}{dt^2} + B_{1}\dfrac{d\varphi}{dt} + A_{2}\varphi
         + \Phi\varphi + {\cal L}\varphi = B_{2}\theta   
         & \mbox{in}\ (0, T), 
 \\[2mm] 
       \theta(0) = \theta_{0},\ \varphi(0) = \varphi_{0},\ \dfrac{d\varphi}{dt}(0) = v_{0},    
     \end{cases}
 \end{equation*}
where $T>0$, $\eta>0$,   
$L : H \to H$ is a linear positive selfadjoint operator, 
$B_{j} : D(B_{j}) \subset H \to H$, $A_{j} : D(A_{j}) \subset H \to H$ ($j = 1, 2$) are 
linear maximal monotone selfadjoint operators, 
$D(A_{j}) \subset V$ ($j=1, 2$), 
$\Phi : D(\Phi) \subset H \to H$ is a maximal monotone operator,   
${\cal L} : H \to H$ is a Lipschitz continuous operator,   
$\theta_{0}, \varphi_{0}, v_{0} \in V$ 
are given.  
Moreover, we deal with the problem 
%
%
%
 \begin{equation*}\tag*{(P)$_{n}$}\label{Pn}
     \begin{cases}
      \delta_{h}\theta_{n} + \eta\delta_{h}\varphi_{n} + A_{1}\theta_{n+1}= 0,  
         \\
       L z_{n+1} + B_{1}v_{n+1} + A_{2}\varphi_{n+1} 
      + \Phi\varphi_{n+1} + {\cal L}\varphi_{n+1} = B_{2}\theta_{n+1}, 
          \\
       z_{0}=z_{1},\ z_{n+1} = \delta_{h}v_{n}, 
         \\
      v_{n+1}=\delta_{h}\varphi_{n}  
     \end{cases}
 \end{equation*}
for $n=0, ..., N-1$,  
where $h=\frac{T}{N}$, $N \in \mathbb{N}$, 
\begin{align}\label{deltah}
\delta_{h}\theta_{n} := \dfrac{\theta_{n+1}-\theta_{n}}{h}, \ 
\delta_{h}\varphi_{n} := \dfrac{\varphi_{n+1}-\varphi_{n}}{h}, \ 
\delta_{h}v_{n} := \dfrac{v_{n+1}-v_n}{h}.   
\end{align}
Here, putting 
\begin{align}
& \widehat{\theta}_{h}(0) := \theta_{0},\  
   \frac{d\widehat{\theta}_{h}}{dt}(t) := \delta_{h}\theta_{n}, \quad
   \widehat{\varphi}_{h}(0) := \varphi_{0},\ 
   \frac{d\widehat{\varphi}_{h}}{dt}(t) := \delta_{h}\varphi_{n},    \label{hat1} 
\\[3mm] 
&\widehat{v}_{h}(0) := v_{0},\   
   \frac{d\widehat{v}_{h}}{dt}(t) := \delta_{h}v_{n}, \label{hat2} 
\\[2mm] 
&\overline{\theta}_{h}(t) := \theta_{n+1},\ \overline{z}_{h}(t) := z_{n+1},\ 
\overline{\varphi}_{h}(t) := \varphi_{n+1},\ \overline{v}_{h}(t) := v_{n+1}  
\label{overandunderline}   
\end{align}
for a.a.\ $t \in (nh, (n+1)h)$, $n=0, ..., N-1$, 
we can rewrite \ref{Pn} as  
%
%
%
 \begin{equation*}\tag*{(P)$_{h}$}\label{Ph}
     \begin{cases}
       \dfrac{d\widehat{\theta}_{h}}{dt} + \eta\dfrac{d\widehat{\varphi}_{h}}{dt}  
       + A_{1}\overline{\theta}_{h}= 0
      \quad \mbox{in}\ (0, T),  
      \\[1.5mm]
         L\overline{z}_{h} + B_{1}\overline{v}_{h} + A_{2}\overline{\varphi}_{h} 
         + \Phi\overline{\varphi}_{h} + {\cal L}\overline{\varphi}_{h} 
         = B_{2}\overline{\theta}_{h} 
     \quad \mbox{in}\ (0, T),  
     \\[0.5mm] 
        \overline{z}_{h} = \dfrac{d\widehat{v}_{h}}{dt},\ 
        \overline{v}_{h} = \dfrac{d\widehat{\varphi}_{h}}{dt} 
     \quad \mbox{in}\ (0, T),   
     \\[1.5mm] 
         \widehat{\theta}_{h}(0)=\theta_{0},\ 
         \widehat{\varphi}_{h}(0) = \varphi_{0},\ \widehat{v}_{h}(0)=v_{0}.     
     \end{cases}
 \end{equation*}
\smallskip

We will assume the following conditions (A1)-(A12): 
%
%
%
 \begin{enumerate} 
 \setlength{\itemsep}{2mm}
 \item[(A1)] 
 $V$ and $H$ are real Hilbert spaces satisfying $V \subset H$ with 
 dense, continuous and compact embedding. Moreover, 
 the inclusions $V \subset H \subset V^{*}$ hold 
 by identifying $H$ with its dual space $H^{*}$, where 
 $V^{*}$ is the dual space of $V$.  
 \item[(A2)] 
 $L : H \to H$ is a bounded linear operator fulfilling 
    \[
    (Lw, z)_{H} = (w, Lz)_{H}\ \mbox{for all}\ w, z \in H, \quad  
    (Lw, w)_{H} \geq c_{L}\|w\|_{H}^2\ \mbox{for all}\ w \in H,   
    \]
 where $c_{L} > 0$ is a constant. 
 \item[(A3)]
 $A_{j} : D(A_{j}) \subset H \to H$ ($j=1, 2$) 
 are linear maximal monotone selfadjoint operators, 
 where $D(A_{j})$ ($j=1, 2$) are linear subspaces of $H$ and $D(A_{j}) \subset V$ ($j=1, 2$).  
 Moreover, there exist bounded linear monotone operators
 $A_{j}^{*} : V \to V^{*}$ ($j=1, 2$) such that  
 \begin{align*}
 &\langle A_{j}^{*}w, z \rangle_{V^{*}, V} 
 = \langle A_{j}^{*}z, w \rangle_{V^{*}, V} 
 \quad\mbox{for all}\ w, z \in V, 
 \\  
 &A_{j}^{*}w = A_{j}w \quad\mbox{for all}\ w \in D(A_{j}). 
 \end{align*}   
 Moreover, for all $\alpha>0$ and for $j=1, 2$ 
 there exists $\omega_{j, \alpha}>0$ such that 
 $$
 \langle A_{j}^{*}w, w \rangle_{V^{*}, V} + \alpha\|w\|_{H}^2 
 \geq \omega_{j, \alpha}\|w\|_{V}^2 
 \quad \mbox{for all}\ w \in V. 
 $$ 
 \item[(A4)] 
 $B : D(B_{j}) \subset H \to H$ ($j=1, 2$) are 
 linear maximal monotone selfadjoint  operators, 
 where $D(B_{j})$ ($j=1, 2$) are linear subspaces of $H$, 
 satisfying 
$D(A_{1}) \subset D(B_{2})$ and 
 \begin{align*}
 &D(B_{1}) \cap D(A_{2}) \neq \emptyset, 
 \\ 
 &(B_{1}w, A_{2}w)_{H} \geq 0 \quad \mbox{for all}\ w \in D(B_{1}) \cap D(A_{2}), 
 \\
 &(B_{2}w, A_{1}w)_{H} \geq 0 \quad \mbox{for all}\ w \in D(A_{1}), 
 \\
 &(B_{1}w, A_{2}z)_{H} = (B_{1}z, A_{2}w)_{H} 
                                    \quad \mbox{for all}\ w, z \in D(B_{1}) \cap D(A_{2}).   
 \end{align*} 
 Moreover, the inclusion $D(A_{2}) \subset V$ holds. 
 \item[(A5)] 
 There exists a constant $C_{A_{1}, B_{2}}>0$ such that 
 $$
 \|B_{2}\theta\|_{H} \leq C_{A_{1}, B_{2}}(\|A_{1}\theta\|_{H} + \|\theta\|_{H})
 \quad \mbox{for all}\ \theta \in D(A_{1}). 
 $$
 \item[(A6)] 
 $\Phi : D(\Phi) \subset H \to H$ is a maximal monotone operator 
 satisfying $\Phi(0) = 0$ and $V \subset D(\Phi)$. 
 Moreover, there exist constants $p, q, C_{\Phi} > 0$ such that  
 \[
 \|\Phi w - \Phi z\|_{H} \leq C_{\Phi}(1 + \|w\|_{V}^p + \|z\|_{V}^{q})\|w-z\|_{V}
 \quad \mbox{for all}\ w, z \in V. 
 \]
 \item[(A7)] 
 There exists a lower semicontinuous convex function  
 $i : V \to \{x \in \mathbb{R}\ |\ x\geq0 \}$ 
 such that $(\Phi w, w-z)_{H} \geq i(w) - i(z)$ for all $w, z \in V$.
 \item[(A8)] 
 $\Phi_{\lambda}(0) = 0$, 
 $(\Phi_{\lambda}w, B_{1}w)_{H} \geq 0$  
 for all $w \in D(B_{1})$, 
 $(\Phi_{\lambda}w, A_{2}w)_{H} \geq 0$  
 for all $w \in D(A_{2})$, 
 where $\lambda > 0$ and $\Phi_{\lambda} : H \to H$ is  
 the Yosida approximation of $\Phi$. 
 \item[(A9)] 
 $B^{*}_{j} : V \to V^{*}$ ($j=1, 2$) are bounded linear monotone operators fulfilling 
 \begin{align*}
 &\langle B^{*}_{j}w, z \rangle_{V^{*}, V} = \langle B^{*}_{j}z, w \rangle_{V^{*}, V}  
 \quad\mbox{for all}\ w, z \in V,  
 \\  
 &B^{*}_{j}w = B_{j}w \quad\mbox{for all}\ w \in D(B_{j}) \cap V. 
 \end{align*}
 \item[(A10)] 
 For all $g \in H$, $a, b, c, d, d' > 0$, $\lambda>0$, if there exists 
 $\varphi_{\lambda} \in V$ such that 
 $$
 L\varphi_{\lambda} + aB_{1}^{*}\varphi_{\lambda} + bA_{2}^{*}\varphi_{\lambda} 
 + c\Phi_{\lambda}\varphi_{\lambda} + d{\cal L}\varphi_{\lambda} 
 + d'B_{2}(I+hA_{1})^{-1}\varphi_{\lambda}  
 = g 
 \quad \mbox{in}\ V^{*},  
 $$ 
 then it follows that $\varphi_{\lambda} \in D(B_{1}) \cap D(A_{2})$ and 
 $$
 L\varphi_{\lambda} + aB_{1}\varphi_{\lambda} + bA_{2}\varphi_{\lambda} 
 + c\Phi_{\lambda}\varphi_{\lambda} + d{\cal L}\varphi_{\lambda} 
 + d'B_{2}(I+hA_{1})^{-1}\varphi_{\lambda}  
 = g 
 \quad \mbox{in}\ H.   
 $$ 
 \item[(A11)] 
 ${\cal L} : H \to H$ is a Lipschitz continuous operator 
 with Lipschitz constant $C_{{\cal L}}>0$.
 \item[(A12)] 
 $\theta_{0} \in D(A_{1})$, $A_{1}\theta_{0} \in V$, 
 $\varphi_{0} \in  D(B_{1}) \cap D(A_{2})$, 
 $v_{0} \in D(B_{1}) \cap V$.      
 \end{enumerate}
We set the conditions (A2) and (A3) 
in reference to \cite[Section 2]{CF1996}.  
The condition (A10) is equivalent to the elliptic regularity theory 
under some cases (see Section \ref{Sec2}). 
Moreover, we set the conditions (A6)-(A8) and (A11)  
by trying to keep typical examples 
of not only the linearized equations of coupled sound and heat flow 
but also some parabolic-hyperbolic phase-field systems 
(see Section \ref{Sec2})  
in reference to assumptions in 
\cite{CK, GP2003, GP2004, GPS2006, WGZ2007, WGZ2007dynamicalBD}.    

\smallskip

\begin{remark}
Owing to  \eqref{hat1}-\eqref{overandunderline}, 
the reader can check directly the following identities: 
\begin{align}
&\|\widehat{\varphi}_h\|_{L^{\infty}(0, T; V)} 
= \max\{\|\varphi_{0}\|_{V}, \|\overline{\varphi}_h\|_{L^{\infty}(0, T; V)}\}, 
\label{rem1}
\\[6mm]
&\|\widehat{v}_h\|_{L^{\infty}(0, T; V)} 
= \max\{\|v_{0}\|_{V}, \|\overline{v}_h\|_{L^{\infty}(0, T; V)}\}, \label{rem2} 
\\[6mm] 
&\|\widehat{\theta}_h\|_{L^{\infty}(0, T; V)} 
= \max\{\|\theta_{0}\|_{V}, 
                           \|\overline{\theta}_h\|_{L^{\infty}(0, T; V)}\}, \label{rem3} 
\\[5mm] 
&\|\overline{\varphi}_h - \widehat{\varphi}_h\|_{L^{\infty}(0, T; V)} 
= h\Bigl\|\frac{d\widehat{\varphi}_h}{dt}\Bigr\|_{L^{\infty}(0, T; V)} 
= h\|\overline{v}_h\|_{L^{\infty}(0, T; V)} , \label{rem4}  
\\[3mm] 
&\|\overline{v}_h - \widehat{v}_h\|_{L^{\infty}(0, T; H)} 
= h\Bigl\|\frac{d\widehat{v}_h}{dt}\Bigr\|_{L^{\infty}(0, T; H)} 
= h\|\overline{z}_h\|_{L^{\infty}(0, T; H)} , \label{rem5} 
\\[3mm] 
&\|\overline{\theta}_h - \widehat{\theta}_h\|_{L^2(0, T; V)}^2 
= \frac{h^2}{3}\Bigl\|\frac{d\widehat{\theta}_h}{dt}\Bigr\|_{L^2(0, T; V)}^2. \label{rem6}
\end{align}  
\end{remark} 
\smallskip

%
%
%
We define solutions of \ref{P} as follows. 
\begin{df}
A pair $(\theta, \varphi)$ with 
\begin{align*}
&\theta \in H^1(0, T; V) \cap L^{\infty}(0, T; V) \cap L^{\infty}(0, T; D(A_{1})),                  
\\ 
&\varphi \in W^{2, \infty}(0, T; H) \cap W^{1, \infty}(0, T; V) 
                                                                     \cap L^2(0, T; D(A_{2})), 
\\[1mm] 
&\frac{d\varphi}{dt} \in L^2(0, T; D(B_{1})),\ \Phi\varphi \in L^{\infty}(0, T; H) 
\end{align*}
is called a solution of \ref{P} if $(\theta, \varphi)$ satisfies 
\begin{align}
&\dfrac{d\theta}{dt} + \eta\dfrac{d\varphi}{dt} + A_{1}\theta = 0   
    \quad\mbox{in}\ H   \quad \mbox{a.e.\ on}\ (0, T), \label{df1} 
\\[2mm] 
&L\dfrac{d^2\varphi}{dt^2} + B_{1}\dfrac{d\varphi}{dt} + A_{2}\varphi
         + \Phi\varphi + {\cal L}\varphi = B_{2}\theta   
         \quad\mbox{in}\ H   \quad \mbox{a.e.\ on}\ (0, T), \label{df2} 
\\ 
&\theta(0) = \theta_{0},\ \varphi(0) = \varphi_{0},\ \dfrac{d\varphi}{dt}(0) = v_{0}                                
        \quad\mbox{in}\ H. \label{df3}
\end{align}
\end{df}

\medskip

Now the main results read as follows. 

\begin{thm}\label{maintheorem1}
Assume that {\rm (A1)-(A12)} hold. 
Then there exists $h_{0} \in (0, 1)$ such that  
for all $h \in (0, h_{0})$ 
there exists a unique solution $(\theta_{n+1}, \varphi_{n+1})$ 
of {\rm \ref{Pn}} 
satisfying 
$$
\theta_{n+1} \in D(A_{1}),\ \varphi_{n+1} \in D(B_{1}) \cap D(A_{2}) 
\quad\mbox{for}\ n=0, ..., N-1.
$$ 
\end{thm}

\begin{thm}\label{maintheorem2}
Assume that {\rm (A1)-(A12)} hold.   
Then there exists a unique solution $(\theta, \varphi)$  
of {\rm \ref{P}}. 
\end{thm}

\begin{thm}\label{erroresti} 
Let $h_{0}$ be as in Theorem \ref{maintheorem1}. 
Assume that {\rm (A1)-(A12)} hold. 
Then there exist constants $h_{00} \in (0, h_{0})$ and 
$M=M(T)>0$ such that 
\begin{align*}
&\|L^{1/2}(\widehat{v}_{h} - v)\|_{L^{\infty}(0, T; H)} 
+ \|B_{1}^{1/2}(\overline{v}_{h}-v)\|_{L^2(0, T; H)} 
+ \|\widehat{\varphi}_{h} - \varphi\|_{L^{\infty}(0, T; V)} 
\\ 
&+ \|\widehat{\theta}_{h} - \theta\|_{L^{\infty}(0, T; H)} 
+ \|\overline{\theta}_{h} - \theta\|_{L^2(0, T; V)} 
\\
&+ \|B_{2}^{1/2}(\widehat{\theta}_{h} - \theta)\|_{L^{\infty}(0, T; H)} 
+ \int_{0}^{T} (B_{2}(\overline{\theta}_{h}(t)-\theta(t)), 
                                A_{1}(\overline{\theta}_{h}(t)-\theta(t)))_{H}\,dt 
\leq M h^{1/2}
\end{align*}
for all $h \in (0, h_{00})$, 
where $v = \frac{d\varphi}{dt}$. 
\end{thm}

This paper is organized as follows. 
In Section \ref{Sec2} we give  
the linearized equations of coupled sound and heat flow 
and some parabolic-hyperbolic phase-field systems as examples. 
In Section \ref{Sec3} we derive existence of solutions to \ref{Pn}. 
In Section \ref{Sec4} we prove that there exists a solution of \ref{P}. 
In Section \ref{Sec5} we establish uniqueness for \ref{P}.  
In Section \ref{Sec6} we obtain error estimates 
between solutions of \ref{P} and solutions of \ref{Ph}.  

\vspace{10pt}

\section{Examples}\label{Sec2}
In this section we give the following examples. 
%
%
%

\begin{example}
We can verify that the problem 
\begin{equation*}\tag*{(P1)}\label{P1}
     \begin{cases}
         \theta_{t} + (\gamma-1)\varphi_{t} - \sigma\Delta\theta = 0      
         & \mbox{in}\ \Omega\times(0, T), 
 \\[1mm]
        \varphi_{tt}  - c^2 \Delta\varphi 
        -m^2 \varphi = -c^2 \Delta\theta 
         & \mbox{in}\ \Omega\times(0, T), 
 \\[1mm] 
         \theta = \varphi = 0 
         & \mbox{on}\ \partial\Omega\times(0, T),  
 \\[1mm]
         \theta(0) = \theta_{0},\ \varphi(0) = \varphi_{0},\ \varphi_{t}(0) = v_{0}  
         &\mbox{in}\ \Omega                                      
     \end{cases}
\end{equation*}
is an example, 
where $c>0$, $\sigma>0$, $m \in \mathbb{R}$, $\gamma>1$, $T>0$ are constants 
and  
$\Omega \subset \mathbb{R}^{3}$ 
is a bounded domain with smooth boundary $\partial\Omega$, 
under the case that 
\[
\theta_{0} \in H^2(\Omega) \cap H_{0}^{1}(\Omega),   
-\Delta\theta_{0} \in H_{0}^{1}(\Omega), 
\varphi_{0} \in H^2(\Omega) \cap H_{0}^{1}(\Omega),  
v_{0} \in H_{0}^{1}(\Omega).   
\]
Indeed, putting 
\begin{align*}
&V:=H_{0}^1(\Omega),\ H:=L^2(\Omega), 
\\ 
&L := I : H \to H, 
\\ 
&A_{1} := - \sigma\Delta : D(A_{1}):=H^2(\Omega) \cap H_{0}^{1}(\Omega)\subset H \to H, 
\\ 
&B_{1} := 0 : D(B_{1}):= H \to H, 
\\ 
&A_{2} := - c^2 \Delta : D(A_{2}):=H^2(\Omega) \cap H_{0}^{1}(\Omega)\subset H \to H, 
\\ 
&B_{2} := -c^2 \Delta : D(B_{2}):=H^2(\Omega) \cap H_{0}^{1}(\Omega)\subset H \to H   
\end{align*}
and defining the operators 
$A_{1}^{*} : V \to V^{*}$, $B_{1}^{*} : V \to V^{*}$, 
$A_{2}^{*} : V \to V^{*}$, $\Phi : D(\Phi) \subset H \to H$, 
${\cal L}: H \to H$, 
$B_{2}^{*} : V \to V^{*}$ 
as 
\begin{align*}
&\langle A_{1}^{*}w, z \rangle_{V^{*}, V} 
:= \sigma\int_{\Omega} \nabla w \cdot \nabla z 
\quad \mbox{for}\ w, z \in V, 
\\[1mm] 
&\langle B_{1}^{*}w, z \rangle_{V^{*}, V} 
:= 0
\quad \mbox{for}\ w, z \in V, 
\\[1mm] 
&\langle A_{2}^{*}w, z \rangle_{V^{*}, V} 
:= c^2 \int_{\Omega} \nabla w \cdot \nabla z 
\quad \mbox{for}\ w, z \in V,  
\\[1mm]
&\Phi z := 0 \quad \mbox{for}\ 
z \in D(\Phi) := H,  
\\[1mm]
&{\cal L}z := -m^2 z \quad \mbox{for}\ z \in H,  
\\[1mm]
&\langle B_{2}^{*}w, z \rangle_{V^{*}, V} 
:= c^2 \int_{\Omega} \nabla w \cdot \nabla z 
\quad \mbox{for}\ w, z \in V,  
\end{align*}
we can check that (A1)--(A12) hold. 
Similarly, we can confirm that 
the homogeneous  Neumann--Neumann problem  
is an example.  
\end{example}

\begin{example}
We see that the problem 
\begin{equation*}\tag*{(P2)}\label{P2}
     \begin{cases}
         \theta_{t} + (\gamma-1)\varphi_{t} - \sigma\Delta\theta = 0     
         & \mbox{in}\ \Omega\times(0, T), 
 \\[1mm]
        \varphi_{tt} + \ep\varphi_{t} - c^2 \Delta\varphi 
        + \beta(\varphi) + \pi(\varphi) = -c^2 \Delta\theta 
         & \mbox{in}\ \Omega\times(0, T), 
 \\[1mm] 
         \theta = \varphi = 0 
         & \mbox{on}\ \partial\Omega\times(0, T),  
 \\[1mm]
         \theta(0) = \theta_{0},\ \varphi(0) = \varphi_{0},\ \varphi_{t}(0) = v_{0}  
         &\mbox{in}\ \Omega                                      
     \end{cases}
\end{equation*}
is an example, 
where $c>0$, $\sigma>0$, $\ep\geq0$, $\gamma>1$, $T>0$ are constants 
and  
$\Omega \subset \mathbb{R}^{3}$ 
is a bounded domain with smooth boundary $\partial\Omega$, 
under the following conditions: 
\begin{enumerate} 
\setlength{\itemsep}{2mm}
\item[(H1)] 
$\beta : \mathbb{R} \to \mathbb{R}$                                
is a single-valued maximal monotone function and 
there exists a proper differentiable (lower semicontinuous) convex function 
$\widehat{\beta} : \mathbb{R} \to [0, +\infty)$ such that 
$\widehat{\beta}(0) = 0$ and 
$\beta(r) = \widehat{\beta}\,'(r) = \partial\widehat{\beta}(r)$ 
for all $r \in \mathbb{R}$, 
where $\widehat{\beta}\,'$ and $\partial\widehat{\beta}$, respectively, 
are the differential and subdifferential of $\widehat{\beta}$.
\item[(H2)] 
$\beta \in C^2(\mathbb{R})$. 
Moreover, there exists a constant $C_{\beta} > 0$ such that 
$|\beta''(r)| \leq C_{\beta}(1+|r|)$ for all $r \in \mathbb{R}$. 
\item[(H3)] 
$\pi : \mathbb{R} \to \mathbb{R}$ is a Lipschitz continuous function. 
\item[(H4)] 
$\theta_{0} \in H^2(\Omega) \cap H_{0}^{1}(\Omega)$,   
$-\Delta\theta_{0} \in H_{0}^{1}(\Omega)$, 
$\varphi_{0} \in H^2(\Omega) \cap H_{0}^{1}(\Omega)$,  
$v_{0} \in H_{0}^{1}(\Omega)$.   
\end{enumerate}
Indeed, putting 
\begin{align*}
&V:=H_{0}^1(\Omega),\ H:=L^2(\Omega), 
\\ 
&L := I : H \to H, 
\\ 
&A_{1} := - \sigma\Delta : D(A_{1}):=H^2(\Omega) \cap H_{0}^{1}(\Omega)\subset H \to H, 
\\ 
&B_{1} := \ep I : D(B_{1}):= H \to H, 
\\ 
&A_{2} := - c^2 \Delta : D(A_{2}):=H^2(\Omega) \cap H_{0}^{1}(\Omega)\subset H \to H, 
\\ 
&B_{2} := -c^2 \Delta : D(B_{2}):=H^2(\Omega) \cap H_{0}^{1}(\Omega)\subset H \to H   
\end{align*}
and defining the operators 
$A_{1}^{*} : V \to V^{*}$, $B_{1}^{*} : V \to V^{*}$, 
$A_{2}^{*} : V \to V^{*}$, $\Phi : D(\Phi) \subset H \to H$, 
${\cal L}: H \to H$, 
$B_{2}^{*} : V \to V^{*}$ 
as 
\begin{align*}
&\langle A_{1}^{*}w, z \rangle_{V^{*}, V} 
:= \sigma\int_{\Omega} \nabla w \cdot \nabla z 
\quad \mbox{for}\ w, z \in V, 
\\[1mm] 
&\langle B_{1}^{*}w, z \rangle_{V^{*}, V} 
:= \ep(w, z)_{H} 
\quad \mbox{for}\ w, z \in V, 
\\[1mm] 
&\langle A_{2}^{*}w, z \rangle_{V^{*}, V} 
:= c^2 \int_{\Omega} \nabla w \cdot \nabla z 
\quad \mbox{for}\ w, z \in V,  
\\[1mm]
&\Phi z := \beta(z) \quad \mbox{for}\ 
z \in D(\Phi) := \{z \in H\ |\ \beta(z) \in H \}, 
\\[1mm]
&{\cal L}z := \pi(z) \quad \mbox{for}\ z \in H,  
\\[1mm]
&\langle B_{2}^{*}w, z \rangle_{V^{*}, V} 
:= c^2 \int_{\Omega} \nabla w \cdot \nabla z 
\quad \mbox{for}\ w, z \in V,  
\end{align*}
we can confirm that (A1)--(A12) hold 
in reference to \cite{K4}. 
Similarly, we can verify that 
the homogeneous  Neumann--Neumann problem  
is an example.  
\end{example}

\begin{example}
We see that the problem 
\begin{equation*}\tag*{(P3)}\label{P3}
     \begin{cases}
         \theta_{t} + (\gamma-1)\varphi_{t} - \sigma\Delta\theta = 0     
         & \mbox{in}\ \Omega\times(0, T), 
 \\[1mm]
        \varphi_{tt} -\ep\Delta\varphi_{t} - c^2 \Delta\varphi 
        + \beta(\varphi) + \pi(\varphi) = -c^2 \Delta\theta 
         & \mbox{in}\ \Omega\times(0, T), 
 \\[1mm] 
         \theta = \varphi = 0 
         & \mbox{on}\ \partial\Omega\times(0, T),  
 \\[1mm]
         \theta(0) = \theta_{0},\ \varphi(0) = \varphi_{0},\ \varphi_{t}(0) = v_{0}  
         &\mbox{in}\ \Omega                                      
     \end{cases}
\end{equation*}
is an example, 
where $c>0$, $\sigma>0$, $\ep\geq0$, $\gamma>1$, $T>0$ are constants 
and  
$\Omega \subset \mathbb{R}^{3}$ 
is a bounded domain with smooth boundary $\partial\Omega$, 
under the three conditions (H1)-(H3) 
and the condition 
\begin{enumerate} 
\item[(H5)] 
$\theta_{0} \in H^2(\Omega) \cap H_{0}^{1}(\Omega)$,   
$-\Delta\theta_{0} \in H_{0}^{1}(\Omega)$, 
$\varphi_{0} \in H^2(\Omega) \cap H_{0}^{1}(\Omega)$,  
$v_{0} \in H^2(\Omega) \cap H_{0}^{1}(\Omega)$.   
\end{enumerate}
Indeed, putting 
\begin{align*}
&V:=H_{0}^1(\Omega),\ H:=L^2(\Omega), 
\\ 
&L := I : H \to H, 
\\ 
&A_{1} := - \sigma\Delta : D(A_{1}):=H^2(\Omega) \cap H_{0}^{1}(\Omega)\subset H \to H, 
\\ 
&B_{1} := -\ep\Delta : D(B_{1}):= H^2(\Omega) \cap H_{0}^{1}(\Omega)\subset H \to H, 
\\ 
&A_{2} := - c^2 \Delta : D(A_{2}):=H^2(\Omega) \cap H_{0}^{1}(\Omega)\subset H \to H, 
\\ 
&B_{2} := -c^2 \Delta : D(B_{2}):=H^2(\Omega) \cap H_{0}^{1}(\Omega)\subset H \to H   
\end{align*}
and defining the operators 
$A_{1}^{*} : V \to V^{*}$, $B_{1}^{*} : V \to V^{*}$, 
$A_{2}^{*} : V \to V^{*}$, $\Phi : D(\Phi) \subset H \to H$, 
${\cal L}: H \to H$, 
$B_{2}^{*} : V \to V^{*}$ 
as 
\begin{align*}
&\langle A_{1}^{*}w, z \rangle_{V^{*}, V} 
:= \sigma\int_{\Omega} \nabla w \cdot \nabla z 
\quad \mbox{for}\ w, z \in V, 
\\[1mm] 
&\langle B_{1}^{*}w, z \rangle_{V^{*}, V} 
:= \ep\int_{\Omega} \nabla w \cdot \nabla z 
\quad \mbox{for}\ w, z \in V, 
\\[1mm] 
&\langle A_{2}^{*}w, z \rangle_{V^{*}, V} 
:= c^2 \int_{\Omega} \nabla w \cdot \nabla z 
\quad \mbox{for}\ w, z \in V,  
\\[1mm]
&\Phi z := \beta(z) \quad \mbox{for}\ 
z \in D(\Phi) := \{z \in H\ |\ \beta(z) \in H \}, 
\\[1mm]
&{\cal L}z := \pi(z) \quad \mbox{for}\ z \in H,  
\\[1mm]
&\langle B_{2}^{*}w, z \rangle_{V^{*}, V} 
:= c^2 \int_{\Omega} \nabla w \cdot \nabla z 
\quad \mbox{for}\ w, z \in V,  
\end{align*}
we can verify that (A1)--(A12) hold 
in reference to \cite{K4}. 
Similarly, we can check that 
the homogeneous  Neumann--Neumann problem  
is an example.  
\end{example}

\begin{example}
The problem 
\begin{equation*}\tag*{(P4)}\label{P4}
     \begin{cases}
         \theta_{t} + \varphi_{t} - \Delta\theta = 0     
         & \mbox{in}\ \Omega\times(0, T), 
 \\[1mm]
        \varphi_{tt} + \varphi_{t} - \Delta\varphi 
        + \beta(\varphi) + \pi(\varphi) = \theta 
         & \mbox{in}\ \Omega\times(0, T), 
 \\[1mm] 
         \theta = \varphi = 0 
         & \mbox{on}\ \partial\Omega\times(0, T),  
 \\[1mm]
         \theta(0) = \theta_{0},\ \varphi(0) = \varphi_{0},\ \varphi_{t}(0) = v_{0}  
         
         &\mbox{in}\ \Omega                                      
     \end{cases}
\end{equation*}
is an example, 
where $\Omega \subset \mathbb{R}^3$ is a bounded domain 
with smooth boundary $\partial\Omega$, $T>0$, 
under the four conditions (H1)-(H4).  
Indeed, putting 
\begin{align*}
&V:=H_{0}^1(\Omega),\ H:=L^2(\Omega), 
\\ 
&L := I : H \to H, 
\\ 
&A_{1} := - \Delta : D(A_{1}):=H^2(\Omega) \cap H_{0}^{1}(\Omega)\subset H \to H, 
\\ 
&B_{1} := I : D(B_{1}):=H \to H, 
\\ 
&A_{2} := - \Delta : D(A_{2}):=H^2(\Omega) \cap H_{0}^{1}(\Omega)\subset H \to H, 
\\ 
&B_{2} := I : D(B_{2}):=H \to H   
\end{align*}
and defining the operators 
$A_{1}^{*} : V \to V^{*}$, $B_{1}^{*} : V \to V^{*}$, 
$A_{2}^{*} : V \to V^{*}$, $\Phi : D(\Phi) \subset H \to H$, 
${\cal L}: H \to H$, 
$B_{2}^{*} : V \to V^{*}$ 
as 
\begin{align*}
&\langle A_{1}^{*}w, z \rangle_{V^{*}, V} 
:= \int_{\Omega} \nabla w \cdot \nabla z 
\quad \mbox{for}\ w, z \in V, 
\\[1mm] 
&\langle B_{1}^{*}w, z \rangle_{V^{*}, V} 
:= (w, z)_{H} 
\quad \mbox{for}\ w, z \in V, 
\\[1mm] 
&\langle A_{2}^{*}w, z \rangle_{V^{*}, V} 
:= \int_{\Omega} \nabla w \cdot \nabla z 
\quad \mbox{for}\ w, z \in V,  
\\[1mm]
&\Phi z := \beta(z) \quad \mbox{for}\ 
z \in D(\Phi) := \{z \in H\ |\ \beta(z) \in H \}, 
\\[1mm]
&{\cal L}z := \pi(z) \quad \mbox{for}\ z \in H,  
\\[1mm]
&\langle B_{2}^{*}w, z \rangle_{V^{*}, V} 
:= (w, z)_{H} 
\quad \mbox{for}\ w, z \in V,  
\end{align*}
we can confirm that (A1)--(A12) hold 
in reference to \cite{K4}. 
Similarly, we can show that 
the homogeneous  Neumann--Neumann problem  
is an example.  
\end{example}
\begin{example}
The problem 
\begin{equation*}\tag*{(P5)}\label{P5}
     \begin{cases}
         \theta_{t} + \varphi_{t} - \Delta\theta = 0     
         & \mbox{in}\ \Omega\times(0, T), 
 \\[1mm]
        \varphi_{tt} - \Delta\varphi_{t} - \Delta\varphi 
        + \beta(\varphi) + \pi(\varphi) = \theta 
         & \mbox{in}\ \Omega\times(0, T), 
 \\[1mm] 
         \theta = \varphi = 0 
         & \mbox{on}\ \partial\Omega\times(0, T),  
 \\[1mm]
         \theta(0) = \theta_{0},\ \varphi(0) = \varphi_{0},\ \varphi_{t}(0) = v_{0} 
         &\mbox{in}\ \Omega                                     
     \end{cases}
\end{equation*}
is an example, 
where $\Omega \subset \mathbb{R}^3$ is a bounded domain 
with smooth boundary $\partial\Omega$, $T>0$, 
under the four conditions (H1)-(H3), (H5). 
Indeed, putting 
\begin{align*}
&V:=H_{0}^1(\Omega),\ H:=L^2(\Omega), 
\\ 
&L := I : H \to H, 
\\ 
&A_{1} := - \Delta : D(A_{1}):=H^2(\Omega) \cap H_{0}^{1}(\Omega)\subset H \to H, 
\\ 
&B_{1} := - \Delta : D(B_{1}):=H^2(\Omega) \cap H_{0}^{1}(\Omega)\subset H \to H, 
\\ 
&A_{2} := - \Delta : D(A_{2}):=H^2(\Omega) \cap H_{0}^{1}(\Omega)\subset H \to H, 
\\ 
&B_{2} := I : D(B_{2}):=H \to H   
\end{align*}
and defining the operators 
$A_{1}^{*} : V \to V^{*}$, $B_{1}^{*} : V \to V^{*}$, 
$A_{2}^{*} : V \to V^{*}$, $\Phi : D(\Phi) \subset H \to H$, 
${\cal L}: H \to H$, 
$B_{2}^{*} : V \to V^{*}$ 
as 
\begin{align*}
&\langle A_{1}^{*}w, z \rangle_{V^{*}, V} 
:= \int_{\Omega} \nabla w \cdot \nabla z 
\quad \mbox{for}\ w, z \in V, 
\\[1mm] 
&\langle B_{1}^{*}w, z \rangle_{V^{*}, V} 
:= \int_{\Omega} \nabla w \cdot \nabla z 
\quad \mbox{for}\ w, z \in V, 
\\[1mm] 
&\langle A_{2}^{*}w, z \rangle_{V^{*}, V} 
:= \int_{\Omega} \nabla w \cdot \nabla z 
\quad \mbox{for}\ w, z \in V,  
\\[1mm]
&\Phi z := \beta(z) \quad \mbox{for}\ 
z \in D(\Phi) := \{z \in H\ |\ \beta(z) \in H \}, 
\\[1mm]
&{\cal L}z := \pi(z) \quad \mbox{for}\ z \in H,  
\\[1mm]
&\langle B_{2}^{*}w, z \rangle_{V^{*}, V} 
:= (w, z)_{H} 
\quad \mbox{for}\ w, z \in V,  
\end{align*}
we can confirm that (A1)--(A12) hold 
in reference to \cite{K4}. 
Similarly, we can show that 
the homogeneous  Neumann--Neumann problem  
is an example.  
\end{example}

\vspace{10pt}

\section{Existence of discrete solutions}\label{Sec3}

In this section we will prove Theorem \ref{maintheorem1}. 

\begin{lem}\label{elliptic2}
There exists 
$h_{1} \in (0, 1)$ 
such that 
\[
0< h_1 < \widetilde{h}:= 
\Bigl(\frac{c_L}{1+C_{\cal L}+\eta C_{A_1, B_2}}
                      +\frac{\eta^2 C_{A_1, B_2}^2}{4(1+C_{\cal L}+\eta C_{A_1, B_2})}
                                                                                                        \Bigr)^{1/2}
                - \frac{\eta C_{A_1, B_2}}{2(1+C_{\cal L}+\eta C_{A_1, B_2})}
\]
and  
for all $g \in H$ and all $h \in (0, h_{1})$ 
there exists a unique solution $\varphi \in D(B_1) \cap D(A_{2})$ 
of the equation 
\[
L\varphi + hB_{1}\varphi + h^2 A_{2}\varphi 
+ h^2 \Phi\varphi + h^2 {\cal L}\varphi + \eta h^2 B_2 (I+hA_1)^{-1}\varphi = g 
\quad \mbox{in}\ H. 
\]
\end{lem}
\begin{proof}
We define the operator $\Psi : V \to V^{*}$ as   
\begin{align*}
\langle \Psi\varphi, w \rangle_{V^{*}, V} 
&:= (L\varphi, w)_{H} + h\langle B_{1}^{*}\varphi, w \rangle_{V^{*}, V} 
+ h^2\langle A_{2}^{*}\varphi, w \rangle_{V^{*}, V} 
+ h^2(\Phi_{\lambda}\varphi, w)_{H} 
\notag \\ 
&\,\quad+ h^2({\cal L}\varphi, w)_{H} 
+ \eta h^2 (B_2 (I+hA_1)^{-1}\varphi, w)_{H}    
\quad \mbox{for}\ \varphi, w \in V. 
\end{align*}
Then this operator $\Psi : V \to V^*$ is monotone, continuous and coercive 
for all $h \in (0, \widetilde{h})$. 
Indeed, since the condition (A5) yields that 
\begin{align}\label{A1B2}
\|B_{2}(I+hA_{1})^{-1}\varphi\|_{H} 
&\leq C_{A_{1}, B_{2}}(\|(I+hA_{1})^{-1}\varphi\|_{H} + \|A_{1}(I+hA_{1})^{-1}\varphi\|_{H})
\notag \\ 
&\leq C_{A_{1}, B_{2}}(1+h^{-1})\|\varphi\|_{H}
\end{align}
for all $\varphi \in H$, 
we derive from (A2), (A3), (A11), the monotonicity of $B_{1}^{*}$ and $\Phi_{\lambda}$, 
and \eqref{A1B2} that 
\begin{align*}
&\langle \Psi\varphi - \Psi\overline{\varphi}, 
                                             \varphi-\overline{\varphi} \rangle_{V^{*}, V} 
\notag \\ 
&= (L(\varphi-\overline{\varphi}), \varphi-\overline{\varphi})_{H} 
    + h\langle B_{1}^{*}(\varphi-\overline{\varphi}), 
                                        \varphi-\overline{\varphi} \rangle_{V^{*}, V} 
    + h^2\langle A_{2}^{*}(\varphi-\overline{\varphi}), 
                                    \varphi-\overline{\varphi} \rangle_{V^{*}, V} 
\notag \\ 
    &\,\quad+ h^2(\Phi_{\lambda}\varphi - \Phi_{\lambda}\overline{\varphi}, 
                                                                      \varphi-\overline{\varphi})_{H} 
    + h^2({\cal L}\varphi - {\cal L}\overline{\varphi}, \varphi-\overline{\varphi})_{H} 
\notag \\ 
&\,\quad + \eta h^2 (B_2 (I+hA_1)^{-1}(\varphi-\overline{\varphi}), 
                                                             (\varphi-\overline{\varphi}))_{H}
\notag \\ 
&\geq c_{L}\|\varphi-\overline{\varphi}\|_{H}^2 
         + \omega_{2, 1}h^2\|\varphi-\overline{\varphi}\|_{V}^2 
         - h^2\|\varphi-\overline{\varphi}\|_{H}^2 
         - C_{{\cal L}}h^2\|\varphi-\overline{\varphi}\|_{H}^2 
\notag \\
&\,\quad- \eta C_{A_1, B_2}(h + h^2) \|\varphi-\overline{\varphi}\|_{H}^2 
\notag \\ 
&\geq \omega_{2, 1}h^2\|\varphi-\overline{\varphi}\|_{V}^2 
\end{align*}
for all $\varphi, \overline{\varphi} \in V$ and all $h \in (0, \widetilde{h})$. 
It follows from the boundedness of the operators 
$L : H \to H$, $B_{1}^* : V \to V^*$, $A_{2}^* : V \to V^*$, 
the Lipschitz continuity of $\Phi_{\lambda} : H \to H$, 
the condition (A11), \eqref{A1B2} 
and the continuity of the embedding $V \hookrightarrow H$ 
that 
there exists a constant $C_{1}=C_{1}(\lambda)>0$ such that  
\begin{align*}
&|\langle \Psi\varphi - \Psi\overline{\varphi}, w \rangle_{V^{*}, V}| 
\notag \\ 
&\leq |(L(\varphi-\overline{\varphi}), w)_{H}| 
    + h|\langle B_{1}^{*}(\varphi-\overline{\varphi}), w \rangle_{V^{*}, V}| 
    + h^2|\langle A_{2}^{*}(\varphi-\overline{\varphi}), w \rangle_{V^{*}, V}| 
\notag \\ 
    &\,\quad+ h^2|(\Phi_{\lambda}\varphi - \Phi_{\lambda}\overline{\varphi}, w)_{H}| 
    + h^2|({\cal L}\varphi - {\cal L}\overline{\varphi}, w)_{H}| 
\notag \\
&\,\quad+ \eta h^2 |(B_2 (I+hA_1)^{-1}(\varphi-\overline{\varphi}), w)_{H}|
\notag \\ 
&\leq C_{1}(1 + h + h^2)\|\varphi-\overline{\varphi}\|_{V}\|w\|_{V}
\end{align*}
for all $\varphi, \overline{\varphi} \in V$ and all $h>0$. 
Moreover, 
the inequality $\langle \Psi\varphi - {\cal L}0, \varphi \rangle_{V^*, V} 
\geq \omega_{2, 1}h^2 \|\varphi\|_{V}^2$ 
holds for all $\varphi \in V$ and all $h \in (0, \widetilde{h})$. 
Therefore the operator $\Psi : V \to V^*$ is surjective for all $h \in (0, \widetilde{h})$ 
(see e.g., \cite[p.\ 37]{Barbu}) 
and then we see from (A10) that for all $g \in H$ and all $h \in (0, \widetilde{h})$ 
there exists a unique solution $\varphi_{\lambda} \in D(B_{1}) \cap D(A_{2})$ 
of the equation 
\begin{align}\label{BJ1}
L\varphi_{\lambda} + hB_{1}\varphi_{\lambda} + h^2A_{2}\varphi_{\lambda} 
+ h^2\Phi_{\lambda}\varphi_{\lambda} + h^2{\cal L}\varphi_{\lambda} 
+ \eta h^2 B_2 (I+hA_1)^{-1}\varphi_{\lambda} 
= g   \quad \mbox{in}\ H. 
\end{align}
Here, multiplying \eqref{BJ1} by $\varphi_{\lambda}$ 
and using the Young inequality, (A11), \eqref{A1B2}, 
we infer that   
\begin{align*}
&(L\varphi_{\lambda}, \varphi_{\lambda})_{H} 
+ h(B_{1}\varphi_{\lambda}, \varphi_{\lambda})_{H} 
+ h^2 \langle A_{2}^{*}\varphi_{\lambda}, \varphi_{\lambda} \rangle_{V^{*}, V}   
+ h^2(\Phi_{\lambda}\varphi_{\lambda}, \varphi_{\lambda})_{H} 
\notag \\ 
&= (g, \varphi_{\lambda})_{H} 
     - h^2({\cal L}\varphi_{\lambda} - {\cal L}0, \varphi_{\lambda})_{H} 
     - h^2({\cal L}0, \varphi_{\lambda})_{H} 
     -\eta h^2 (B_2 (I + h A_1)^{-1}\varphi_{\lambda}, \varphi_{\lambda})_{H} 
\notag \\ 
&\leq \frac{c_{L}}{2}\|\varphi_{\lambda}\|_{H}^2 
         + \frac{1}{2c_{L}}\|g\|_{H}^2 
         + C_{{\cal L}}h^2\|\varphi_{\lambda}\|_{H}^2 
         + \frac{\|{\cal L}0\|_{H}^2}{2}h^2 
         + \frac{1}{2}h^2\|\varphi_{\lambda}\|_{H}^2 
\notag \\
&\,\quad + \eta C_{A_1, B_2}(h+h^2)\|\varphi_{\lambda}\|_{H}^2,      
\end{align*}
whence the conditions (A2) and (A3), 
the monotonicity of $B_{1}$ and $\Phi_{\lambda}$ 
imply that 
there exists $h_{1} \in (0, \min\{1, \widetilde{h}\})$ such that 
for all $h \in (0, h_{1})$ there exists a constant $C_{2}=C_{2}(h)>0$ satisfying 
\begin{align}\label{BJ2}
\|\varphi_{\lambda}\|_{V}^2 \leq C_{2}   
\end{align}
for all $\lambda>0$. 
We have from \eqref{BJ1}, (A8), \eqref{A1B2} and the Young inequality that 
\begin{align*}
&h^2\|\Phi_{\lambda}\varphi_{\lambda}\|_{H}^2 
\notag \\ 
&= (g, \Phi_{\lambda}\varphi_{\lambda})_{H} 
     - (L\varphi_{\lambda}, \Phi_{\lambda}\varphi_{\lambda})_{H} 
    - h(B_{1}\varphi_{\lambda}, \Phi_{\lambda}\varphi_{\lambda})_{H} 
    - h^2(A_{2}\varphi_{\lambda}, \Phi_{\lambda}\varphi_{\lambda})_{H} 
\notag \\ 
&\,\quad- h^2({\cal L}\varphi_{\lambda}, \Phi_{\lambda}\varphi_{\lambda})_{H} 
     -\eta h^2 (B_2 (I + h A_1)^{-1}\varphi_{\lambda}, 
                                                        \Phi_{\lambda}\varphi_{\lambda})_{H} 
\notag \\ 
&\leq \frac{2}{h^2}\|g\|_{H}^2 + \frac{2}{h^2}\|L\varphi_{\lambda}\|_{H}^2 
         + 2 h^2\|{\cal L}\varphi_{\lambda}\|_{H}^2 
         + 2 \eta^2 C_{A_1, B_2}^2 (1+h)^2 \|\varphi_{\lambda}\|_{H}^2 
         + \frac{1}{2}h^2\|\Phi_{\lambda}\varphi_{\lambda}\|_{H}^2.    
\end{align*}
Thus, owing to the boundedness of the operator $L : H \to H$, 
(A11) and \eqref{BJ2}, 
it holds that for all $h \in (0, h_{1})$ 
there exists a constant $C_{3}=C_{3}(h)>0$ such that 
\begin{align}\label{BJ3}
\|\Phi_{\lambda}\varphi_{\lambda}\|_{H}^2 \leq C_{3}   
\end{align}
for all $\lambda>0$. 
The equation \eqref{BJ1} yields that 
\begin{align*}
h\|B_{1}\varphi_{\lambda}\|_{H}^2 
&= (g, B_{1}\varphi_{\lambda})_{H} 
     - (L\varphi_{\lambda}, B_{1}\varphi_{\lambda})_{H} 
    - h^2(A_{2}\varphi_{\lambda}, B_{1}\varphi_{\lambda})_{H} 
    - h^2(\Phi_{\lambda}\varphi_{\lambda}, B_{1}\varphi_{\lambda})_{H} 
\notag \\ 
&\,\quad- h^2({\cal L}\varphi_{\lambda}, B_{1}\varphi_{\lambda})_{H} 
-\eta h^2 (B_2 (I + h A_1)^{-1}\varphi_{\lambda}, B_{1}\varphi_{\lambda})_{H},  
\end{align*}
and hence we deduce from the boundedness of the operator $L : H \to H$, 
(A4), (A8), (A11), \eqref{A1B2}, the Young inequality and \eqref{BJ2} that 
for all $h \in (0, h_{1})$ there exists a constant $C_{4}=C_{4}(h)>0$ satisfying 
\begin{align}\label{BJ4}
\|B_1 \varphi_{\lambda}\|_{H}^2 \leq C_{4}(h)   
\end{align}
for all $\lambda>0$. 
We derive from \eqref{A1B2}-\eqref{BJ4} that for all $h \in (0, h_{1})$ 
there exists a constant $C_{5}=C_{5}(h)>0$ such that  
\begin{align}\label{BJ5}
\|A_{2}\varphi_{\lambda}\|_{H}^2 \leq C_{5}(h)   
\end{align}
for all $\lambda>0$. 
Hence the inequalities \eqref{BJ2}-\eqref{BJ5} mean that 
there exist $\varphi \in D(B_{1}) \cap D(A_{2})$ and $\xi \in H$ such that 
\begin{align}
\label{ellipweak1} 
&\varphi_{\lambda} \to \varphi \quad \mbox{weakly in}\ V,  \\ 
\label{ellipweak1'} 
&L\varphi_{\lambda} \to L\varphi \quad \mbox{weakly in}\ H,  \\ 
\label{ellipweak2}
&\Phi_{\lambda}(\varphi_{\lambda}) \to \xi  
                                             \quad \mbox{weakly in}\ H,  \\ 
\label{ellipweak3}
&B_{1}\varphi_{\lambda} \to B_{1}\varphi \quad \mbox{weakly in}\ H,  \\ 
\label{ellipweak4}
&A_{2}\varphi_{\lambda} \to A_{2}\varphi 
                                              \quad \mbox{weakly in}\ H 
\end{align}
as $\lambda = \lambda_{j} \to +0$. 
Here it follows from \eqref{BJ2}, \eqref{ellipweak1}, the compact of 
the embedding $V \hookrightarrow H$ that 
\begin{align}\label{ellipstrong} 
\varphi_{\lambda} \to \varphi \quad \mbox{strongly in}\ H
\end{align}
as $\lambda = \lambda_{j} \to +0$. 
Also, we see from \eqref{ellipweak2} and \eqref{ellipstrong} that  
$(\Phi_{\lambda}\varphi_{\lambda}, \varphi_{\lambda})_{H} 
\to (\xi, \varphi)_{H}$  
as $\lambda = \lambda_{j} \to +0$. 
Thus the inclusion and the identity 
\begin{align}\label{BJ6}
\varphi \in D(\Phi),\ \xi = \Phi\varphi 
\end{align}
hold (see e.g., \cite[Lemma 1.3, p.\ 42]{Barbu1}).   

Therefore, thanks to \eqref{BJ1}, \eqref{ellipweak1'}-\eqref{BJ6} and (A11), 
we can verify that 
there exists a solution $\varphi \in D(B_{1}) \cap D(A_{2})$ of the equation 
\[
L\varphi + hB_{1}\varphi + h^2 A_{2}\varphi 
+ h^2 \Phi\varphi + h^2 {\cal L}\varphi + \eta h^2 B_2 (I + h A_1)^{-1}\varphi = g 
\quad \mbox{in}\ H. 
\] 
Moreover, 
the solution $\varphi$ of this problem is unique 
by (A2), (A3), the monotonicity of $B_{1}$ and $\Phi$, (A11) and \eqref{A1B2}.    
\end{proof}

\begin{prth1.1}
Let $h_{1}$ be as in Lemma \ref{elliptic2} and let $h \in (0, h_{1})$. 
Then we infer from \eqref{deltah}, the linearity of 
the operators $A_{1}$, $L$, $B_{1}$, $B_{2}$ and $A_{2}$ that 
the problem \ref{Pn} can be written as 
\begin{equation*}\tag*{(Q)$_{n}$}\label{Qn}
     \begin{cases}
      \theta_{n+1} + hA_{1}\theta_{n+1} 
      = \theta_{n} + \eta(\varphi_{n} - \varphi_{n+1}), 
      \\[3mm] 
      L\varphi_{n+1} + hB_{1}\varphi_{n+1} + h^2A_{2}\varphi_{n+1} 
      + h^2\Phi\varphi_{n+1} + h^2{\cal L}\varphi_{n+1} 
      + \eta h^2 B_2 (I + h A_1)^{-1} \varphi_{n+1}
      \\ 
      = L\varphi_{n} + hLv_{n} + hB_{1}\varphi_{n} 
         + h^2 B_2 (I + h A_1)^{-1} (\eta\varphi_{n}+\theta_{n})   
     \end{cases}
\end{equation*}
and then proving Theorem \ref{maintheorem1} is equivalent to 
show existence and uniqueness of solutions to \ref{Qn} for $n=0, ..., N-1$. 
It suffices to consider the case that $n=0$. 
Owing to Lemma \ref{elliptic2}, 
there exists a unique solution $\varphi_{1} \in D(B_{1}) \cap D(A_{2})$ 
of the equation 
\begin{align*}
&L\varphi_{1} + hB_{1}\varphi_{1} + h^2A_{2}\varphi_{1} 
      + h^2\Phi\varphi_{1} + h^2{\cal L}\varphi_{1} 
      + \eta h^2 B_2 (I + h A_1)^{-1} \varphi_{1} 
\\ 
&= L\varphi_{0} + hLv_{0} + hB_{1}\varphi_{0} 
         + h^2 B_2 (I + h A_1)^{-1} (\eta\varphi_{0}+\theta_{0}).    
\end{align*}
Therefore, putting $\theta_{1}:=(I + hA_{1})^{-1}(\theta_{0}+\eta(\varphi_{0}-\varphi_{1}))$, 
we can conclude that 
there exists a unique solution $(\theta_{1}, \varphi_{1})$ of \ref{Qn} 
in the case that $n=0$. 
\qed  
\end{prth1.1}

\vspace{10pt}
 
\section{Uniform estimates for \ref{Ph} and passage to the limit}\label{Sec4}

In this section we will derive a priori estimates for \ref{Ph}
and will show Theorem \ref{maintheorem2} 
by passing to the limit in \ref{Ph} as $h\to+0$.

\begin{lem}\label{lemkuri1}
Let $h_{0}$ be as in Theorem \ref{maintheorem1}. 
Then there exist constants 
$h_{2} \in (0, h_{0})$ and $C=C(T)>0$ such that  
\begin{align*}
&\|\overline{v}_{h}\|_{L^{\infty}(0, T; H)}^2 
+ h\|\overline{z}_{h}\|_{L^2(0, T; H)}^2 
+ \|B_{1}^{1/2}\overline{v}_{h}\|_{L^2(0, T; H)}^2   
+ \|\overline{\varphi}_{h}\|_{L^{\infty}(0, T; V)}^2 
\\
&+ h\|\overline{v}_{h}\|_{L^2(0, T; V)}^2  
+ \|B_{2}^{1/2}\overline{\theta}_{h}\|_{L^{\infty}(0, T; H)}^2 
+ h\Bigl\|B_{2}^{1/2}\frac{d\widehat{\theta}_{h}}{dt}\Bigr\|_{L^2(0, T; H)}^2 
\leq C 
\end{align*}
for all $h \in (0, h_{2})$. 
\end{lem}
\begin{proof}
We test the second equation in \ref{Pn} by $hv_{n+1}$ ($= \varphi_{n+1}-\varphi_{n}$) 
and recall \eqref{deltah} to obtain that 
\begin{align}\label{a1}
&(L(v_{n+1}-v_{n}), v_{n+1})_{H} + h\|B_{1}^{1/2}v_{n+1}\|_{H}^2  
+ \langle A_{2}^{*}\varphi_{n+1}, \varphi_{n+1}-\varphi_{n} \rangle_{V^{*}, V}  
\notag \\ 
&+ (\varphi_{n+1}, \varphi_{n+1}-\varphi_{n})_{H} 
+ (\Phi\varphi_{n+1}, \varphi_{n+1}-\varphi_{n})_{H} 
\notag \\ 
&= h(B_{2}\theta_{n+1}, v_{n+1})_{H} - h({\cal L}\varphi_{n+1}, v_{n+1})_{H} 
     + h(\varphi_{n+1}, v_{n+1})_{H}. 
\end{align}
Here it holds that 
\begin{align}\label{a2}
&(L(v_{n+1}-v_{n}), v_{n+1})_{H} 
\notag \\ 
&= (L^{1/2}(v_{n+1}-v_{n}), L^{1/2}v_{n+1})_{H} 
\notag \\ 
&=\frac{1}{2}\|L^{1/2}v_{n+1}\|_{H}^2 - \frac{1}{2}\|L^{1/2}v_{n}\|_{H}^2 
    + \frac{1}{2}\|L^{1/2}(v_{n+1}-v_{n})\|_{H}^2  
\end{align}
and 
\begin{align}\label{a3}
&\langle A_{2}^{*}\varphi_{n+1}, \varphi_{n+1}-\varphi_{n} \rangle_{V^{*}, V}  
+ (\varphi_{n+1}, \varphi_{n+1}-\varphi_{n})_{H} 
\notag \\ 
&= \frac{1}{2}\langle A_{2}^{*}\varphi_{n+1}, \varphi_{n+1} \rangle_{V^{*}, V} 
     - \frac{1}{2}\langle A_{2}^{*}\varphi_{n}, \varphi_{n} \rangle_{V^{*}, V} 
     + \frac{1}{2}\langle A_{2}^{*}(\varphi_{n+1}-\varphi_{n}), 
                                               \varphi_{n+1}-\varphi_{n} \rangle_{V^{*}, V}  
\notag \\ 
&\,\quad + \frac{1}{2}\|\varphi_{n+1}\|_{H}^2 - \frac{1}{2}\|\varphi_{n}\|_{H}^2 
    + \frac{1}{2}\|\varphi_{n+1}-\varphi_{n}\|_{H}^2. 
\end{align}
The first equation in \ref{Pn} yields that 
\begin{align}\label{a3'}
&h(B_{2}\theta_{n+1}, v_{n+1})_{H} 
\notag \\
&= \frac{h}{\eta}\Bigl(B_{2}\theta_{n+1}, 
                                  -\frac{\theta_{n+1}-\theta_{n}}{h}-A_{1}\theta_{n+1} \Bigr)_{H} 
\notag \\ 
&=-\frac{1}{2\eta}
         \Bigl(\|B_{2}^{1/2}\theta_{n+1}\|_{H}^2 - \|B_{2}^{1/2}\theta_{n}\|_{H}^2 
                  + \|B_{2}^{1/2}(\theta_{n+1}-\theta_{n})\|_{H}^2 \Bigr) 
\notag \\ 
&\,\quad -\frac{h}{\eta}(B_{2}\theta_{n+1}, A_{1}\theta_{n+1})_{H}. 
\end{align}
Thus it follows from \eqref{a1}-\eqref{a3'}, (A4), (A7), (A11), 
the continuity of the embedding $V \hookrightarrow H$ 
and the Young inequality that 
there exist constants $C_{1}, C_{2} > 0$ such that 
\begin{align}\label{a4}
&\frac{1}{2}\|L^{1/2}v_{n+1}\|_{H}^2 - \frac{1}{2}\|L^{1/2}v_{n}\|_{H}^2 
    + \frac{1}{2}\|L^{1/2}(v_{n+1}-v_{n})\|_{H}^2 
    + h\|B_{1}^{1/2}v_{n+1}\|_{H}^2  
\notag \\
&+ \frac{1}{2}\langle A_{2}^{*}\varphi_{n+1}, \varphi_{n+1} \rangle_{V^{*}, V} 
     - \frac{1}{2}\langle A_{2}^{*}\varphi_{n}, \varphi_{n} \rangle_{V^{*}, V} 
     + \frac{1}{2}\langle A_{2}^{*}(\varphi_{n+1}-\varphi_{n}), 
                                               \varphi_{n+1}-\varphi_{n} \rangle_{V^{*}, V}  
\notag \\ 
&+ \frac{1}{2}\|\varphi_{n+1}\|_{H}^2 - \frac{1}{2}\|\varphi_{n}\|_{H}^2 
    + \frac{1}{2}\|\varphi_{n+1}-\varphi_{n}\|_{H}^2 
   + i(\varphi_{n+1}) - i(\varphi_{n}) 
\notag \\ 
& + \frac{1}{2\eta}\|B_{2}^{1/2}\theta_{n+1}\|_{H}^2 
             - \frac{1}{2\eta}\|B_{2}^{1/2}\theta_{n}\|_{H}^2 
             + \frac{1}{2\eta}\|B_{2}^{1/2}(\theta_{n+1}-\theta_{n})\|_{H}^2
\notag \\
&\leq h\|v_{n+1}\|_{H}^2 
        + C_{1}h\|\varphi_{n+1}\|_{V}^2 + C_{2}h
\end{align}
for all $h \in (0, h_{0})$. 
Moreover, summing \eqref{a4} over $n = 0, ..., m-1$ with $1 \leq m \leq N$ 
leads to the inequality 
\begin{align}\label{a7}
&\frac{1}{2}\|L^{1/2}v_{m}\|_{H}^2 
+ \frac{1}{2}\sum_{n=0}^{m-1}\|L^{1/2}(v_{n+1}-v_{n})\|_{H}^2 
+ h\sum_{n=0}^{m-1}\|B_{1}^{1/2}v_{n+1}\|_{H}^2  
+ \frac{1}{2}\langle A_{2}^{*}\varphi_{m}, \varphi_{m} \rangle_{V^{*}, V} 
\notag \\
&+ \frac{1}{2}\|\varphi_{m}\|_{H}^2 
+ \frac{1}{2}\sum_{n=0}^{m-1}\langle A_{2}^{*}(\varphi_{n+1}-\varphi_{n}), 
                                               \varphi_{n+1}-\varphi_{n} \rangle_{V^{*}, V} 
+ \frac{1}{2}\sum_{n=0}^{m-1}\|\varphi_{n+1}-\varphi_{n}\|_{H}^2   
\notag \\
&+ i(\varphi_{m}) + \frac{1}{2\eta}\|B_{2}^{1/2}\theta_{m}\|_{H}^2 
+ \frac{1}{2\eta}\sum_{n=0}^{m-1}\|B_{2}^{1/2}(\theta_{n+1}-\theta_{n})\|_{H}^2
\notag \\ 
&\leq \frac{1}{2}\|L^{1/2}v_{0}\|_{H}^2 
   + \frac{1}{2}\langle A_{2}^{*}\varphi_{0}, \varphi_{0} \rangle_{V^{*}, V} 
   + \frac{1}{2}\|\varphi_{0}\|_{H}^2 + i(\varphi_{0}) 
   + \frac{1}{2\eta}\|B_{2}^{1/2}\theta_{0}\|_{H}^2 
\notag \\
&\,\quad+ h\sum_{n=0}^{m-1}\|v_{n+1}\|_{H}^2 
        + C_{1}h\sum_{n=0}^{m-1}\|\varphi_{n+1}\|_{V}^2 + C_{2}T  
\end{align}
for all $h \in (0, h_{0})$. 
Here we see from (A3) that 
\begin{align}\label{a8}
\frac{1}{2}\langle A_{2}^{*}\varphi_{m}, \varphi_{m} \rangle_{V^{*}, V} 
+ \frac{1}{2}\|\varphi_{m}\|_{H}^2 
\geq \frac{\omega_{2, 1}}{2}\|\varphi_{m}\|_{V}^2 
\end{align}
and 
\begin{align}\label{a9}
&\frac{1}{2}\sum_{n=0}^{m-1}\langle A_{2}^{*}(\varphi_{n+1}-\varphi_{n}), 
                                              \varphi_{n+1}-\varphi_{n} \rangle_{V^{*}, V} 
+ \frac{1}{2}\sum_{n=0}^{m-1}\|\varphi_{n+1}-\varphi_{n}\|_{H}^2   
\notag \\ 
&\geq \frac{\omega_{2, 1}}{2}\sum_{n=0}^{m-1}\|\varphi_{n+1}-\varphi_{n}\|_{V}^2 
= \frac{\omega_{2, 1}}{2}h^2\sum_{n=0}^{m-1}\|v_{n+1}\|_{V}^2.  
\end{align}
Thus we derive from \eqref{a7}-\eqref{a9} and (A2) that 
\begin{align*}
&\left(\frac{c_{L}}{2}-h\right)\|v_{m}\|_{H}^2 
+ \frac{c_{L}}{2}h^2\sum_{n=0}^{m-1}\|z_{n+1}\|_{H}^2 
+ h\sum_{n=0}^{m-1}\|B_{1}^{1/2}v_{n+1}\|_{H}^2  
+ \left(\frac{\omega_{2, 1}}{2} - C_{1}h \right)\|\varphi_{m}\|_{V}^2 
\notag \\
&+ \frac{\omega_{2, 1}}{2}h^2\sum_{n=0}^{m-1}\|v_{n+1}\|_{V}^2   
+ \frac{1}{2\eta}\|B_{2}^{1/2}\theta_{m}\|_{H}^2 
+ \frac{1}{2\eta}h^2\sum_{n=0}^{m-1}\|B_{2}^{1/2}\delta_{h}\theta_{n}\|_{H}^2 
\notag \\
&\leq \frac{1}{2}\|L^{1/2}v_{0}\|_{H}^2 
   + \frac{1}{2}\langle A_{2}^{*}\varphi_{0}, \varphi_{0} \rangle_{V^{*}, V} 
   + \frac{1}{2}\|\varphi_{0}\|_{H}^2 + i(\varphi_{0}) 
   + \frac{1}{2\eta}\|B_{2}^{1/2}\theta_{0}\|_{H}^2 
\notag \\
&\,\quad+ h\sum_{j=0}^{m-1}\|v_{j}\|_{H}^2 
        + C_{1}h\sum_{j=0}^{m-1}\|\varphi_{j}\|_{V}^2 + C_{2}T, 
\end{align*}
whence there exist constants $h_{2} \in (0, h_{0})$ 
and $C_{3} = C_{3}(T) > 0$ such that 
\begin{align}\label{a12}
&\|v_{m}\|_{H}^2 
+ h^2\sum_{n=0}^{m-1}\|z_{n+1}\|_{H}^2 
+ h\sum_{n=0}^{m-1}\|B_{1}^{1/2}v_{n+1}\|_{H}^2  
\notag \\ 
&+ \|\varphi_{m}\|_{V}^2 
+ h^2\sum_{n=0}^{m-1}\|v_{n+1}\|_{V}^2   
+ \|B_{2}^{1/2}\theta_{m}\|_{H}^2 
+ h^2\sum_{n=0}^{m-1}\|B_{2}^{1/2}\delta_{h}\theta_{n}\|_{H}^2
\notag \\
&\leq C_{3}h\sum_{j=0}^{m-1}\|v_{j}\|_{H}^2 
        + C_{3}h\sum_{j=0}^{m-1}\|\varphi_{j}\|_{V}^2 + C_{3}
\end{align}
for all $h \in (0, h_{2})$. 
Therefore it follows from the inequality \eqref{a12} and 
the discrete Gronwall lemma (see e.g., \cite[Prop.\ 2.2.1]{Jerome}) that 
there exists a constant $C_{4} = C_{4}(T) > 0$ such that   
\begin{align*}
&\|v_{m}\|_{H}^2 
+ h^2\sum_{n=0}^{m-1}\|z_{n+1}\|_{H}^2 
+ h\sum_{n=0}^{m-1}\|B_{1}^{1/2}v_{n+1}\|_{H}^2  
\notag \\ 
&+ \|\varphi_{m}\|_{V}^2 
+ h^2\sum_{n=0}^{m-1}\|v_{n+1}\|_{V}^2   
+ \|B_{2}^{1/2}\theta_{m}\|_{H}^2 
+ h^2\sum_{n=0}^{m-1}\|B_{2}^{1/2}\delta_{h}\theta_{n}\|_{H}^2 
\leq  C_{4}
\end{align*}
for all $h \in (0, h_{2})$ and $m=1,..., N$. 
\end{proof}

\begin{lem}\label{lemkuri3}
Let $h_{2}$ be as in Lemma \ref{lemkuri1}. 
Then there exists a constant $C=C(T)>0$ such that  
\begin{align*}
\|z_{1}\|_{H}^2 + h\|B_{1}^{1/2}z_{1}\|_{H}^2 
+ \|v_{1}\|_{V}^2 + h^2\|z_{1}\|_{V}^2 
+ \langle B_{2}^* (\eta v_{1} + A_{1}\theta_{1}), \eta v_{1} + A_{1}\theta_{1} \rangle_{V^*, V} 
\leq C 
\end{align*}
for all $h\in(0, h_{2})$. 
\end{lem}
\begin{proof}
The second equation in \ref{Pn}, 
the identities $v_{1} = v_{0} + hz_{1}$ and $\varphi_{1} = \varphi_{0} + hv_{1}$ 
yield that  
\begin{align}\label{coco1}
Lz_{1} + B_{1}v_{0} + hB_{1}z_{1} + A_{2}\varphi_{0} + hA_{2}v_{1} 
+ \Phi\varphi_{1} + {\cal L}\varphi_{1} = B_{2}\theta_{1}. 
\end{align}
Then we test \eqref{coco1} by $z_{1}$ to infer that 
\begin{align}\label{coco2}
&\|L^{1/2}z_{1}\|_{H}^2 + (B_{1}v_{0}, z_{1})_{H} 
+ h(B_{1}z_{1}, z_{1})_{H} + (A_{2}\varphi_{0}, z_{1})_{H} 
+ h(A_{2}v_{1}, z_{1})_{H}   
\notag \\ 
&+ (\Phi\varphi_{1}, z_{1})_{H} + ({\cal L}\varphi_{1}, z_{1})_{H} 
= (B_{2}\theta_{1}, z_{1})_{H}. 
\end{align}
Here we derive from (A3) that 
\begin{align}\label{coco3}
h(A_{2}v_{1}, z_{1})_{H} 
&= (A_{2}v_{1}, v_{1} - v_{0})_{H} 
= \langle A_{2}^{*}v_{1}, v_{1} - v_{0} \rangle_{V^{*}, V}
\notag \\ 
&= \frac{1}{2}\langle A_{2}^{*}v_{1}, v_{1} \rangle_{V^{*}, V}
     - \frac{1}{2}\langle A_{2}^{*}v_{0}, v_{0} \rangle_{V^{*}, V} 
\notag \\ 
  &\,\quad+ \frac{1}{2}\langle A_{2}^{*}(v_{1} - v_{0}), v_{1} - v_{0} \rangle_{V^{*}, V} 
\notag \\ 
&\geq \frac{\omega_{2, 1}}{2}\|v_{1}\|_{V}^2 - \frac{1}{2}\|v_{1}\|_{H}^2 
     - \frac{1}{2}\langle A_{2}^{*}v_{0}, v_{0} \rangle_{V^{*}, V}  
\notag \\ 
&\,\quad + \frac{\omega_{2, 1}}{2}\|v_{1} - v_{0}\|_{V}^2 
             - \frac{1}{2}\|v_{1} - v_{0}\|_{H}^2.   
\end{align}
We see from (A6) and Lemma \ref{lemkuri1} that 
there exists a constant $C_{1} = C_{1}(T) > 0$ such that 
\begin{align}\label{coco4}
|(\Phi\varphi_{1}, z_{1})_{H}| 
\leq C_{\Phi}(1 + \|\varphi_{1}\|_{V}^{p})\|\varphi_{1}\|_{V}\|z_{1}\|_{H} 
\leq C_{1}\|z_{1}\|_{H}.  
\end{align}
Also, the first equation in \ref{Pn} 
and the identity $v_{1}-v_{0}=hz_{1}$ 
imply that 
\begin{align}\label{coco4'}
&\frac{1}{2\eta}
   \langle B_{2}^* (\eta v_{1} + A_{1}\theta_{1}), \eta v_{1} + A_{1}\theta_{1} \rangle_{V^*, V} 
- \frac{1}{2\eta}
   \langle B_{2}^* (\eta v_{0} + A_{1}\theta_{0}), \eta v_{0} + A_{1}\theta_{0} \rangle_{V^*, V} 
\notag \\ 
&+ \frac{1}{2\eta}
     \langle B_{2}^* (\eta (v_{1}-v_{0}) + A_{1}(\theta_{1}-\theta_{0})), 
                                   \eta (v_{1}-v_{0}) + A_{1}(\theta_{1}-\theta_{0}) \rangle_{V^*, V} 
\notag \\ 
&= \frac{1}{\eta}
        \langle B_{2}^* (\eta v_{1} + A_{1}\theta_{1}), 
                                   \eta (v_{1}-v_{0}) + A_{1}(\theta_{1}-\theta_{0}) \rangle_{V^*, V} 
\notag \\ 
&= - (B_{2}\theta_{1}, z_{1})_{H} + (B_{2}\theta_{0}, z_{1})_{H} 
    - \frac{1}{\eta h}(B_{2}(\theta_{1}-\theta_{0}), A_{1}(\theta_{1}-\theta_{0}))_{H}.   
\end{align}
Hence it follows from \eqref{coco2}-\eqref{coco4'}, (A2), (A4) 
and the monotonicity of $B_{2}^{*} : V \to V^*$ 
that 
\begin{align}\label{coco5}
&c_{L}\|z_{1}\|_{H}^2 + h\|B_{1}^{1/2}z_{1}\|_{H}^2 
+ \frac{\omega_{2, 1}}{2}\|v_{1}\|_{V}^2 + \frac{\omega_{2, 1}}{2}h^2\|z_{1}\|_{V}^2 
\notag \\ 
&+ \frac{1}{2\eta}
   \langle B_{2}^* (\eta v_{1} + A_{1}\theta_{1}), \eta v_{1} + A_{1}\theta_{1} \rangle_{V^*, V} 
\notag \\ 
&\leq - (B_{1}v_{0}, z_{1})_{H} - (A_{2}\varphi_{0}, z_{1})_{H} 
        + \frac{1}{2}\|v_{1}\|_{H}^2 
        + \frac{1}{2}\langle A_{2}^{*}v_{0}, v_{0} \rangle_{V^{*}, V}  
        + \frac{1}{2}\|v_{1} - v_{0}\|_{H}^2  
\notag \\ 
    &\,\quad+ C_{1}\|z_{1}\|_{H} - ({\cal L}\varphi_{1}, z_{1})_{H} + (B_{2}\theta_{0}, z_{1})_{H}
+ \frac{1}{2\eta}
   \langle B_{2}^* (\eta v_{0} + A_{1}\theta_{0}), \eta v_{0} + A_{1}\theta_{0} \rangle_{V^*, V}. 
\end{align}
Thus we deduce from 
\eqref{coco5}, (A11), the Young inequality and Lemma \ref{lemkuri1} 
that Lemma \ref{lemkuri3} holds. 
\end{proof}

\begin{lem}\label{lemkuri4}
Let $h_{2}$ be as in Lemma \ref{lemkuri1}. 
Then there exist constants $h_{3} \in (0, h_{2})$ and $C=C(T)>0$ such that 
\begin{align*}
\|\overline{z}_{h}\|_{L^{\infty}(0, T; H)}^2  
+ \|B_{1}^{1/2}\overline{z}_{h}\|_{L^2(0, T; H)}^2  
+ \|\overline{v}_{h}\|_{L^{\infty}(0, T; V)}^2 
+ h\|\overline{z}_{h}\|_{L^2(0, T; V)}^2  
\leq C 
\end{align*}
for all $h \in (0, h_{3})$. 
\end{lem}
\begin{proof}
Let $n \in \{1, ..., N-1\}$. 
Then we have from the second equation in \ref{Pn} that 
\begin{align*}
&L(z_{n+1} - z_{n}) + hB_{1}z_{n+1} + hA_{2}v_{n+1} 
+ \Phi\varphi_{n+1} - \Phi\varphi_{n} 
+ {\cal L}\varphi_{n+1} - {\cal L}\varphi_{n} 
\notag \\ 
&= B_{2} (\theta_{n+1} - \theta_{n}).   
\end{align*} 
Here, since it holds that  
\begin{align*}
&(L(z_{n+1}-z_{n}), z_{n+1})_{H} 
=(L^{1/2}(z_{n+1}-z_{n}), L^{1/2}z_{n+1})_{H} 
\notag \\ 
&=\frac{1}{2}\|L^{1/2}z_{n+1}\|_{H}^2 - \frac{1}{2}\|L^{1/2}z_{n}\|_{H}^2 
+ \frac{1}{2}\|L^{1/2}(z_{n+1} - z_{n})\|_{H}^2, 
\end{align*}
we see that 
\begin{align}\label{pasta1}
&\frac{1}{2}\|L^{1/2}z_{n+1}\|_{H}^2 - \frac{1}{2}\|L^{1/2}z_{n}\|_{H}^2 
+ \frac{1}{2}\|L^{1/2}(z_{n+1} - z_{n})\|_{H}^2 
+ h\|B_{1}^{1/2}z_{n+1}\|_{H}^2 
\notag \\ 
&+ \langle A_{2}^{*}v_{n+1}, v_{n+1} - v_{n} \rangle_{V^{*}, V} 
  + (v_{n+1}, v_{n+1} - v_{n})_{H} 
\notag \\ 
&= -h\left(\frac{\Phi\varphi_{n+1} - \Phi\varphi_{n}}{h}, z_{n+1} \right)_{H} 
    - h\left(\frac{{\cal L}\varphi_{n+1} - {\cal L}\varphi_{n}}{h}, z_{n+1} \right)_{H} 
\notag \\ 
&\,\quad + (B_{2}(\theta_{n+1}-\theta_{n}), z_{n+1})_{H}
             + h(v_{n+1}, z_{n+1})_{H}. 
\end{align}
On the other hand, the identity  
\begin{align}\label{pasta2}
&\langle A_{2}^{*}v_{n+1}, v_{n+1} - v_{n} \rangle_{V^{*}, V} 
  + (v_{n+1}, v_{n+1} - v_{n})_{H} 
\notag \\ 
&= \frac{1}{2}\langle A_{2}^{*}v_{n+1}, v_{n+1} \rangle_{V^{*}, V} 
     - \frac{1}{2}\langle A_{2}^{*}v_{n}, v_{n} \rangle_{V^{*}, V} 
     + \frac{1}{2}\langle A_{2}^{*}(v_{n+1} - v_{n}), v_{n+1} - v_{n} \rangle_{V^{*}, V} 
\notag \\ 
&\,\quad + \frac{1}{2}\|v_{n+1}\|_{H}^2 - \frac{1}{2}\|v_{n}\|_{H}^2 
             + \frac{1}{2}\|v_{n+1}-v_{n}\|_{H}^2  
\end{align}
holds. 
The condition (A6) and Lemma \ref{lemkuri1} mean that 
there exists a constant $C_{1}=C_{1}(T)>0$ such that 
\begin{align}\label{pasta3}
-h\left(\frac{\Phi\varphi_{n+1} - \Phi\varphi_{n}}{h}, z_{n+1} \right)_{H} 
&\leq C_{\Phi}h(1 + \|\varphi_{n+1}\|_{V}^{p} + \|\varphi_{n}\|_{V}^{q})
                                                                          \|v_{n+1}\|_{V}\|z_{n+1}\|_{H} 
\notag \\ 
&\leq C_{1}h\|v_{n+1}\|_{V}\|z_{n+1}\|_{H} 
\end{align}
for all $h \in (0, h_{2})$. 
Also, the first equation in \ref{Pn} 
and the identity $v_{n+1}-v_{n}=hz_{n+1}$ 
yield that 
\begin{align}\label{pasta3'}
&\frac{1}{2\eta}
   \langle B_{2}^* (\eta v_{n+1} + A_{1}\theta_{n+1}), 
                                                     \eta v_{n+1} + A_{1}\theta_{n+1} \rangle_{V^*, V} 
\notag \\ 
&- \frac{1}{2\eta}
       \langle B_{2}^* (\eta v_{n} + A_{1}\theta_{n}), 
                                                       \eta v_{n} + A_{1}\theta_{n} \rangle_{V^*, V} 
\notag \\ 
&+ \frac{1}{2\eta}
     \langle B_{2}^* (\eta (v_{n+1}-v_{n}) + A_{1}(\theta_{n+1}-\theta_{n})), 
                            \eta (v_{n+1}-v_{n}) + A_{1}(\theta_{n+1}-\theta_{n}) \rangle_{V^*, V} 
\notag \\ 
&= \frac{1}{\eta}
        \langle B_{2}^* (\eta v_{n+1} + A_{1}\theta_{n+1}), 
                          \eta (v_{n+1}-v_{n}) + A_{1}(\theta_{n+1}-\theta_{n}) \rangle_{V^*, V} 
\notag \\ 
&= - (B_{2}(\theta_{n+1} - \theta_{n}), z_{n+1})_{H} 
    - \frac{1}{\eta h}(B_{2}(\theta_{n+1} - \theta_{n}), A_{1}(\theta_{n+1} - \theta_{n}))_{H}.   
\end{align}
Hence it follows from \eqref{pasta1}-\eqref{pasta3'}, (A4) and (A11) 
that there exists a constant $C_{2}=C_{2}(T)>0$ such that 
\begin{align}\label{pasta4}
&\frac{1}{2}\|L^{1/2}z_{n+1}\|_{H}^2 - \frac{1}{2}\|L^{1/2}z_{n}\|_{H}^2 
+ \frac{1}{2}\|L^{1/2}(z_{n+1} - z_{n})\|_{H}^2 
+ h\|B_{1}^{1/2}z_{n+1}\|_{H}^2 
\notag \\ 
&+ \frac{1}{2}\langle A_{2}^{*}v_{n+1}, v_{n+1} \rangle_{V^{*}, V} 
     - \frac{1}{2}\langle A_{2}^{*}v_{n}, v_{n} \rangle_{V^{*}, V} 
     + \frac{1}{2}\langle A_{2}^{*}(v_{n+1} - v_{n}), v_{n+1} - v_{n} \rangle_{V^{*}, V} 
\notag \\ 
&+ \frac{1}{2}\|v_{n+1}\|_{H}^2 - \frac{1}{2}\|v_{n}\|_{H}^2 
             + \frac{1}{2}\|v_{n+1}-v_{n}\|_{H}^2 
\notag \\ 
&+ \frac{1}{2\eta}
   \langle B_{2}^* (\eta v_{n+1} + A_{1}\theta_{n+1}), 
                                                     \eta v_{n+1} + A_{1}\theta_{n+1} \rangle_{V^*, V} 
\notag \\ 
&- \frac{1}{2\eta}
       \langle B_{2}^* (\eta v_{n} + A_{1}\theta_{n}), 
                                                       \eta v_{n} + A_{1}\theta_{n} \rangle_{V^*, V} 
\notag \\ 
&+ \frac{1}{2\eta}
     \langle B_{2}^* (\eta (v_{n+1}-v_{n}) + A_{1}(\theta_{n+1}-\theta_{n})), 
                            \eta (v_{n+1}-v_{n}) + A_{1}(\theta_{n+1}-\theta_{n}) \rangle_{V^*, V} 
\notag \\ 
&\leq C_{2}h\|v_{n+1}\|_{V}\|z_{n+1}\|_{H} 
\end{align}
for all $h \in (0, h_{2})$. 
Then we sum \eqref{pasta4} over $n=1, ..., \ell-1$ with $2 \leq \ell \leq N$ 
to derive that  
\begin{align*}
&\frac{1}{2}\|L^{1/2}z_{\ell}\|_{H}^2 
+ \frac{1}{2}\sum_{n=1}^{\ell-1}\|L^{1/2}(z_{n+1} - z_{n})\|_{H}^2 
+ h\sum_{n=1}^{\ell-1}\|B_{1}^{1/2}z_{n+1}\|_{H}^2 
\notag \\ 
&+ \frac{1}{2}\langle A_{2}^{*}v_{\ell}, v_{\ell} \rangle_{V^{*}, V} 
     + \frac{1}{2}\sum_{n=1}^{\ell-1}
                 \langle A_{2}^{*}(v_{n+1} - v_{n}), v_{n+1} - v_{n} \rangle_{V^{*}, V} 
\notag \\ 
&+ \frac{1}{2}\|v_{\ell}\|_{H}^2  
             + \frac{1}{2}\sum_{n=1}^{\ell-1}\|v_{n+1}-v_{n}\|_{H}^2 
+ \frac{1}{2\eta}
   \langle B_{2}^* (\eta v_{\ell} + A_{1}\theta_{\ell}), 
                                                     \eta v_{\ell} + A_{1}\theta_{\ell} \rangle_{V^*, V} 
\notag \\ 
&\leq \frac{1}{2}\|L^{1/2}z_{1}\|_{H}^2 
+ \frac{1}{2}\langle A_{2}^{*}v_{1}, v_{1} \rangle_{V^{*}, V} 
+ \frac{1}{2}\|v_{1}\|_{H}^2  
+ \frac{1}{2\eta}
   \langle B_{2}^* (\eta v_{1} + A_{1}\theta_{1}), 
                                                     \eta v_{1} + A_{1}\theta_{1} \rangle_{V^*, V}
\notag \\
&\,\quad+ C_{2}h\sum_{n=0}^{\ell-1}\|v_{n+1}\|_{V}\|z_{n+1}\|_{H}.  
\end{align*}
Thus we have from (A2) and (A3) that 
\begin{align}\label{pasta6}
&\frac{c_{L}}{2}\|z_{\ell}\|_{H}^2 
+ h\sum_{n=1}^{\ell-1}\|B_{1}^{1/2}z_{n+1}\|_{H}^2 
+ \frac{\omega_{2, 1}}{2}\|v_{\ell}\|_{V}^2  
             + \frac{\omega_{2, 1}}{2}h^2\sum_{n=1}^{\ell-1}\|z_{n+1}\|_{V}^2 
\notag \\ 
&+ \frac{1}{2\eta}
       \langle B_{2}^* (\eta v_{\ell} + A_{1}\theta_{\ell}), 
                                                     \eta v_{\ell} + A_{1}\theta_{\ell} \rangle_{V^*, V} 
\notag \\ 
&\leq \frac{1}{2}\|L^{1/2}z_{1}\|_{H}^2 
+ \frac{1}{2}\langle A_{2}^{*}v_{1}, v_{1} \rangle_{V^{*}, V} 
+ \frac{1}{2}\|v_{1}\|_{H}^2  
+ \frac{1}{2\eta}
     \langle B_{2}^* (\eta v_{1} + A_{1}\theta_{1}), 
                                                     \eta v_{1} + A_{1}\theta_{1} \rangle_{V^*, V} 
\notag \\
&\,\quad+ C_{2}h\sum_{n=0}^{\ell-1}\|v_{n+1}\|_{V}\|z_{n+1}\|_{H} 
\end{align}
for all $h \in (0, h_{2})$ and $\ell = 2, ..., N$. 
Therefore we infer from \eqref{pasta6}, 
the boundedness of $L$ and $A_{2}^*$, 
and Lemma \ref{lemkuri3} that 
there exists a constant $C_{3}=C_{3}(T)>0$ such that 
\begin{align}\label{pasta7}
&\frac{c_{L}}{2}\|z_{m}\|_{H}^2  
+ h\sum_{n=0}^{m-1}\|B_{1}^{1/2}z_{n+1}\|_{H}^2 
+ \frac{\omega_{2, 1}}{2}\|v_{m}\|_{V}^2  
             + \frac{\omega_{2, 1}}{2}h^2\sum_{n=0}^{m-1}\|z_{n+1}\|_{V}^2 
\notag \\ 
&\leq C_{3} + C_{2}h\sum_{n=0}^{m-1}\|v_{n+1}\|_{V}\|z_{n+1}\|_{H} 
\end{align}
for all $h \in (0, h_{2})$ and $m=1, ..., N$. 
Moreover, we see from \eqref{pasta7} and the Young inequality that 
\begin{align}\label{pasta8}
&\frac{1}{2}(c_{L} - C_{2}h)\|z_{m}\|_{H}^2  
+ h\sum_{n=0}^{m-1}\|B_{1}^{1/2}z_{n+1}\|_{H}^2 
+ \frac{1}{2}(\omega_{2, 1} - C_{2}h)\|v_{m}\|_{V}^2  
\notag \\ 
&+ \frac{\omega_{2, 1}}{2}h^2\sum_{n=0}^{m-1}\|z_{n+1}\|_{V}^2 
\leq C_{3} + \frac{C_{2}}{2}h\sum_{j=0}^{m-1}\|v_{j}\|_{V}^2  
+ \frac{C_{2}}{2}h\sum_{j=0}^{m-1}\|z_{j}\|_{H}^2  
\end{align}
for all $h \in (0, h_{2})$ and $m=1, ..., N$. 
Hence there exist constants $h_{3} \in (0, h_{2})$ and $C_{4}=C_{4}(T)>0$ such that 
\begin{align*}
&\|z_{m}\|_{H}^2 + h\sum_{n=0}^{m-1}\|B_{1}^{1/2}z_{n+1}\|_{H}^2 
+ \|v_{m}\|_{V}^2 + h^2\sum_{n=0}^{m-1}\|z_{n+1}\|_{V}^2 
\notag \\ 
&\leq C_{4} + C_{4}h\sum_{j=0}^{m-1}\|v_{j}\|_{V}^2  
+ C_{4}h\sum_{j=0}^{m-1}\|z_{j}\|_{H}^2  
\end{align*}
for all $h \in (0, h_{3})$ and $m=1, ..., N$. 
Therefore, 
owing to the discrete Gronwall lemma (see e.g., \cite[Prop.\ 2.2.1]{Jerome}), 
there exists a constant $C_{5}=C_{5}(T)>0$ satisfying 
\begin{align*}
\|z_{m}\|_{H}^2 + h\sum_{n=0}^{m-1}\|B_{1}^{1/2}z_{n+1}\|_{H}^2 
+ \|v_{m}\|_{V}^2 + h^2\sum_{n=0}^{m-1}\|z_{n+1}\|_{V}^2 
\leq C_{5} 
\end{align*}
for all $h \in (0, h_{3})$ and $m=1, ..., N$. 
\end{proof}

\begin{lem}\label{lemkuri5}
Let $h_{2}$ be as in Lemma \ref{lemkuri1}. 
Then there exists a constant $C=C(T)>0$ such that 
\begin{align*}
\|\Phi\overline{\varphi}_{h}\|_{L^{\infty}(0, T; H)}  
\leq C 
\end{align*}
for all $h \in (0, h_{2})$. 
\end{lem}
\begin{proof}
We can obtain this lemma by (A6) and Lemma \ref{lemkuri1}. 
\end{proof}

\begin{lem}\label{lemkuri2}
Let $h_{3}$ be as in Lemma \ref{lemkuri4}. 
Then there exist constants $h_{4} \in (0, h_{3})$ and $C=C(T)>0$ such that  
\begin{align*}
\Bigl\|\dfrac{d\widehat{\theta}_{h}}{dt}\Bigr\|_{L^2(0, T; H)}^2 
+ h\Bigl\|\dfrac{d\widehat{\theta}_{h}}{dt}\Bigr\|_{L^2(0, T; V)}^2 
+ \|\overline{\theta}_{h}\|_{L^{\infty}(0, T; V)}^2 
\leq C 
\end{align*}
for all $h \in (0, h_{4})$. 
\end{lem}
\begin{proof}
We multiply the first equation in \ref{Pn} 
by $\theta_{n+1}-\theta_{n}$ 
and by $h\theta_{n+1}$, respectively, 
and use the Young inequality to obtain that 
\begin{align}\label{hoho1}
&h\left\|\frac{\theta_{n+1}-\theta_{n}}{h}\right\|_{H}^2 
+ \langle A_{1}^{*}\theta_{n+1}, \theta_{n+1}-\theta_{n} \rangle_{V^{*}, V} 
+ (\theta_{n+1}-\theta_{n}, \theta_{n+1})_{H} 
+ h(A_{1}\theta_{n+1}, \theta_{n+1})_{H} 
\notag \\ 
&= -\eta h\left(v_{n+1}, \frac{\theta_{n+1}-\theta_{n}}{h}\right)_{H} 
     -\eta h (v_{n+1}, \theta_{n+1})_{H} 
\notag \\ 
&\leq  \eta^2 h\|v_{n+1}\|_{H}^2  
+ \frac{1}{2}h\left\|\frac{\theta_{n+1}-\theta_{n}}{h}\right\|_{H}^2 
+ \frac{1}{2}h\|\theta_{n+1}\|_{H}^2.  
\end{align}
Here it holds that 
\begin{align}\label{hoho2}
&\langle A_{1}^{*}\theta_{n+1}, \theta_{n+1} - \theta_{n} \rangle_{V^{*}, V} 
  + (\theta_{n+1}-\theta_{n}, \theta_{n+1})_{H} 
\notag \\ 
&= \frac{1}{2}\langle A_{1}^{*}\theta_{n+1}, \theta_{n+1} \rangle_{V^{*}, V} 
     - \frac{1}{2}\langle A_{1}^{*}\theta_{n}, \theta_{n} \rangle_{V^{*}, V} 
     + \frac{1}{2}\langle A_{1}^{*}(\theta_{n+1} - \theta_{n}), 
                                                           \theta_{n+1} - \theta_{n} \rangle_{V^{*}, V} 
\notag \\ 
&\,\quad + \frac{1}{2}\|\theta_{n+1}\|_{H}^2 - \frac{1}{2}\|\theta_{n}\|_{H}^2 
             + \frac{1}{2}\|\theta_{n+1}-\theta_{n}\|_{H}^2.   
\end{align}
Thus we see from \eqref{hoho1}, \eqref{hoho2} 
and the continuity of the embedding $V \hookrightarrow H$ 
that 
there exists a constant $C_{1}>0$ such that  
\begin{align}\label{hoho3}
&\frac{1}{2}h\left\|\frac{\theta_{n+1}-\theta_{n}}{h}\right\|_{H}^2 
+ \frac{1}{2}\langle A_{1}^{*}\theta_{n+1}, \theta_{n+1} \rangle_{V^{*}, V} 
     - \frac{1}{2}\langle A_{1}^{*}\theta_{n}, \theta_{n} \rangle_{V^{*}, V} 
\notag \\ 
&+ \frac{1}{2}\langle A_{1}^{*}(\theta_{n+1}-\theta_{n}), 
                                                  \theta_{n+1}-\theta_{n} \rangle_{V^{*}, V} 
+ \frac{1}{2}\|\theta_{n+1}\|_{H}^2 - \frac{1}{2}\|\theta_{n}\|_{H}^2 
             + \frac{1}{2}\|\theta_{n+1}-\theta_{n}\|_{H}^2
\notag \\ 
&\leq \eta^2 h\|v_{n+1}\|_{H}^2 + C_{1}h\|\theta_{n+1}\|_{V}^2  
\end{align} 
for all $h \in (0, h_{3})$. 
Therefore we can prove Lemma \ref{lemkuri2} 
by summing \eqref{hoho3} over $n=0, ..., m-1$ with $1 \leq m \leq N$, 
the condition (A3), Lemma \ref{lemkuri1} 
and the discrete Gronwall lemma (see e.g., \cite[Prop.\ 2.2.1]{Jerome}).     
\end{proof}

\begin{lem}\label{lemkuri6}
Let $h_{4}$ be as in Lemma \ref{lemkuri2}. 
Then there exists a constant $C=C(T)>0$ such that  
\begin{align*} 
\Bigl\|\frac{d\widehat{\theta}_{h}}{dt}\Bigr\|_{L^2(0, T; V)}^2 
+ \|A_{1}\overline{\theta}_{h}\|_{L^{\infty}(0, T; H)}^2 
\leq C 
\end{align*}
for all $h \in (0, h_{4})$. 
\end{lem}
\begin{proof}
It follows from the first equation in \ref{Pn} that 
\begin{align*}
&h\Bigl\langle A_{1}^* \frac{\theta_{n+1}-\theta_{n}}{h}, 
                                      \frac{\theta_{n+1}-\theta_{n}}{h} \Bigr\rangle_{V^*, V} 
+ h\Bigl\|\frac{\theta_{n+1}-\theta_{n}}{h}\Bigr\|_{H}^2 
\\ 
&+ \frac{1}{2}\|A_{1}\theta_{n+1}\|_{H}^2 - \frac{1}{2}\|A_{1}\theta_{n}\|_{H}^2 
  + \frac{1}{2}\|A_{1}(\theta_{n+1}-\theta_{n})\|_{H}^2 
\\ 
&= -\eta h\Bigl\langle A_{1}^* \frac{\theta_{n+1}-\theta_{n}}{h}, 
                                                                      v_{n+1} \Bigr\rangle_{V^*, V} 
+ h\Bigl\|\frac{\theta_{n+1}-\theta_{n}}{h}\Bigr\|_{H}^2  
\end{align*}
and then we can prove this lemma 
by (A3), the boundedness of the operator $A_{1}^{*} : V \to V^*$, 
the Young inequality, Lemma \ref{lemkuri4}, 
summing over $n=0, ..., m-1$ with $1 \leq m \leq N$ 
and Lemma \ref{lemkuri2}.  
\end{proof}

\begin{lem}\label{lemkuri7}
Let $h_{4}$ be as in Lemma \ref{lemkuri2}. 
Then there exists a constant $C=C(T)>0$ such that 
\begin{align*}
\|B_{2}\overline{\theta}_{h}\|_{L^{\infty}(0, T; H)}^2  
+ \|B_{1}\overline{v}_{h}\|_{L^2(0, T; H)}^2 
+ \|A_{2}\overline{\varphi}_{h}\|_{L^2(0, T; H)}^2 
\leq C 
\end{align*}
for all $h \in (0, h_{4})$.
\end{lem}
\begin{proof}
The condition (A5), Lemmas \ref{lemkuri2} and \ref{lemkuri6} mean that 
there exists a constant $C_{1}=C_{1}(T)>0$ such that 
\begin{align}\label{pen0}
\|B_{2}\overline{\theta}_{h}\|_{L^{\infty}(0, T; H)}^2  \leq C_{1}
\end{align}
for all $h \in (0, h_{4})$. 
The second equation in \ref{Pn} yields that 
\begin{align*}
&h\|B_{1}v_{n+1}\|_{H}^2 
= h(B_{1}v_{n+1}, B_{1}v_{n+1})_{H} 
\notag \\ 
&=-h(Lz_{n+1}, B_{1}v_{n+1})_{H} 
    - h(A_{2}\varphi_{n+1}, B_{1}v_{n+1})_{H} 
    - h(\Phi\varphi_{n+1}, B_{1}v_{n+1})_{H} 
\notag \\ 
&\,\quad- h({\cal L}\varphi_{n+1}, B_{1}v_{n+1})_{H} 
    + h(B_{2}\theta_{n+1}, B_{1}v_{n+1})_{H}    
\end{align*}
and then we derive from the Young inequality, 
the boundedness of the operator $L : H \to H$, 
(A11) 
and Lemma \ref{lemkuri1} 
that 
there exists a constant $C_{2}=C_{2}(T)>0$ satisfying 
\begin{align}\label{pen1}
h\|B_{1}v_{n+1}\|_{H}^2  
&\leq C_{2}h\|z_{n+1}\|_{H}^2 - h(A_{2}\varphi_{n+1}, B_{1}v_{n+1})_{H}  
       + C_{2}h\|\Phi\varphi_{n+1}\|_{H}^2 
\notag \\ 
&\,\quad+ C_{2}h\|B_{2}\theta_{n+1}\|_{H}^2 + C_{2}h 
\end{align}
for all $h \in (0, h_{4})$. 
Here we have from (A4) that 
\begin{align}\label{pen2}
-h(A_{2}\varphi_{n+1}, B_{1}v_{n+1})_{H}  
&= -(A_{2}\varphi_{n+1}, B_{1}\varphi_{n+1}-B_{1}\varphi_{n})_{H} 
\notag \\ 
&= - \frac{1}{2}(A_{2}\varphi_{n+1}, B_{1}\varphi_{n+1})_{H} 
     + \frac{1}{2}(A_{2}\varphi_{n}, B_{1}\varphi_{n})_{H} 
\notag \\ 
&\,\quad- \frac{1}{2}
                      (A_{2}(\varphi_{n+1}-\varphi_{n}), B_{1}(\varphi_{n+1}-\varphi_{n}))_{H}.
\end{align}
Thus summing \eqref{pen1} over $n=0, ..., m-1$ with $1 \leq m \leq N$ 
and using \eqref{pen0}, \eqref{pen2}, 
Lemmas \ref{lemkuri4} and \ref{lemkuri5} imply that 
there exists a constant $C_{3} = C_{3}(T) > 0$ such that 
\begin{align}\label{pen3}
\|B_{1}\overline{v}_{h}\|_{L^2(0, T; H)}^2 \leq C_{3}
\end{align}
for all $h \in (0, h_{4})$. 
Moreover, we deduce from the second equation in \ref{Ph}, \eqref{pen0}, \eqref{pen3}, 
Lemmas \ref{lemkuri4} and \ref{lemkuri5}, 
(A11) and Lemma \ref{lemkuri1} that 
there exists a constant $C_{4} = C_{4}(T) > 0$ satisfying  
\begin{align*}
\|A_{2}\overline{\varphi}_{h}\|_{L^2(0, T; H)}^2 \leq C_{4}
\end{align*}
for all $h \in (0, h_{4})$.   
\end{proof}

\begin{lem}\label{lemkuri8}
Let $h_{4}$ be as in Lemma \ref{lemkuri2}. 
Then there exists a constant $C=C(T)>0$ such that 
\begin{align*}
&\|\widehat{\varphi}_{h}\|_{W^{1, \infty}(0, T; V)}  
+ \|\widehat{v}_{h}\|_{W^{1, \infty}(0, T; H)}  
\notag \\ 
&+ \|\widehat{v}_{h}\|_{L^{\infty}(0, T; V)}  
+ \|\widehat{\theta}_{h}\|_{H^1(0, T; V)}  
+ \|\widehat{\theta}_{h}\|_{L^{\infty}(0, T; V)}   
\leq C 
\end{align*}
for all $h \in (0, h_{4})$. 
\end{lem}
\begin{proof}
We can check this lemma by 
\eqref{rem1}-\eqref{rem3}, 
Lemmas \ref{lemkuri1}, \ref{lemkuri4}, \ref{lemkuri2} and \ref{lemkuri6}. 
\end{proof}

\begin{prth1.2e}
By Lemmas \ref{lemkuri1}, \ref{lemkuri4}-\ref{lemkuri8}, 
and \eqref{rem4}-\eqref{rem6}, 
there exist some functions   
\begin{align*}
&\theta \in H^1(0, T; V) \cap L^{\infty}(0, T; V) \cap L^{\infty}(0, T; D(A_{1})),  
\\ 
&\varphi \in L^{\infty}(0, T; V) \cap L^2(0, T; D(A_{2})), 
\\ 
&\xi \in L^{\infty}(0, T; H)  
\end{align*} 
such that 
\begin{align*}
\frac{d\varphi}{dt} \in L^{\infty}(0, T; V) \cap L^2(0, T; D(B_{1})),\ 
\frac{d^2\varphi}{dt^2} \in L^{\infty}(0, T; H)  
\end{align*}
and 
\begin{align}
\label{weak1} 
&\widehat{\varphi}_{h} \to \varphi 
\quad \mbox{weakly$^{*}$ in}\ W^{1, \infty}(0, T; V),  
\\[2.5mm] 
\notag 
&\overline{v}_{h} \to \frac{d\varphi}{dt}     
\quad \mbox{weakly$^{*}$ in}\ L^{\infty}(0, T; V),  
\\[2.5mm] 
\label{weak2}
&\widehat{v}_{h} \to \frac{d\varphi}{dt} 
\quad \mbox{weakly$^{*}$ in}\ W^{1, \infty}(0, T; H) \cap L^{\infty}(0, T; V),  
\\[2.5mm] 
\notag 
&\overline{z}_{h} \to \frac{d^2\varphi}{dt^2}   
\quad \mbox{weakly$^{*}$ in}\ L^{\infty}(0, T; H),  
\\[2.5mm] 
\label{weak3}
&L\overline{z}_{h} \to L\frac{d^2\varphi}{dt^2}   
\quad \mbox{weakly$^{*}$ in}\ L^{\infty}(0, T; H),  
\\[2.5mm] 
\label{weak4}
&\widehat{\theta}_{h} \to \theta  
\quad \mbox{weakly$^{*}$ in}\ H^1(0, T; V) \cap L^{\infty}(0, T; V),   
\\[2.5mm] 
\notag 
&\overline{\varphi}_{h} \to \varphi    
\quad \mbox{weakly$^{*}$ in}\ L^{\infty}(0, T; V),  
\\[2.5mm] 
\notag 
&\overline{\theta}_{h} \to \theta     
\quad \mbox{weakly$^{*}$ in}\ L^{\infty}(0, T; V),  
\\[2.5mm] 
\label{weak5}
&A_{1}\overline{\theta}_{h} \to A_{1}\theta     
\quad \mbox{weakly$^{*}$ in}\ L^{\infty}(0, T; H),  
\\[2.5mm]  
\label{weak6}
&B_{1}\overline{v}_{h} \to B_{1}\frac{d\varphi}{dt}      
\quad \mbox{weakly in}\ L^2(0, T; H),  
\\[2.5mm] 
\label{weak7}
&A_{2}\overline{\varphi}_{h} \to A_{2}\varphi      
\quad \mbox{weakly in}\ L^2(0, T; H),  
\\[2.5mm] 
\label{weak8}
&\Phi\overline{\varphi}_{h} \to \xi      
\quad \mbox{weakly$^{*}$ in}\ L^{\infty}(0, T; H), 
\\[2.5mm]
\label{weak9}
&B_{2}\overline{\theta}_{h} \to B_{2}\theta 
\quad \mbox{weakly$^{*}$ in}\ L^{\infty}(0, T; H)   
\end{align}
as $h=h_{j}\to+0$. 
Here, since we infer from Lemma \ref{lemkuri8}, 
the compactness of the embedding $V \hookrightarrow H$ and 
the convergence \eqref{weak1} 
that   
\begin{align}\label{stronghatvarphi}
\widehat{\varphi}_{h} \to \varphi  
\quad \mbox{strongly in}\ C([0, T]; H)
\end{align}
as $h=h_{j}\to+0$ (see e.g., \cite[Section 8, Corollary 4]{Simon}), 
we see from \eqref{rem4} and Lemma \ref{lemkuri4} that   
\begin{align}\label{strongoverlinevarphi}
\overline{\varphi}_{h} \to \varphi  
\quad \mbox{strongly in}\ L^{\infty}(0, T; H)
\end{align}
as $h=h_{j}\to+0$. 
Hence the convergences \eqref{weak8} and \eqref{strongoverlinevarphi} yield that   
\begin{align*}
\int_{0}^{T}(\Phi\overline{\varphi}_{h}(t), \overline{\varphi}_{h}(t))_{H}\,dt 
\to \int_{0}^{T}(\xi(t), \varphi(t))_{H}\,dt 
\end{align*}
as $h=h_{j}\to+0$ 
and then it holds that  
\begin{align}\label{xiPhivarphi}
\xi = \Phi\varphi  \quad\mbox{in}\ H\ \mbox{a.e.\ on}\ (0, T)  
\end{align}
(see e.g., \cite[Lemma 1.3, p.\ 42]{Barbu1}).  
On the other hand, it follows from Lemma \ref{lemkuri8}, 
the compactness of the embedding $V \hookrightarrow H$ and \eqref{weak4} 
that 
\begin{align}\label{stronghattheta}
\widehat{\theta}_{h} \to \theta  
\quad \mbox{strongly in}\ C([0, T]; H)
\end{align}
as $h=h_{j}\to+0$. 
Similarly, we derive from \eqref{weak2} that 
\begin{align}\label{stronghatv}
\widehat{v}_{h} \to \frac{d\varphi}{dt}  
\quad \mbox{strongly in}\ C([0, T]; H)
\end{align}
as $h=h_{j}\to+0$. 
Therefore, 
combining \eqref{weak1}, \eqref{weak3}-\eqref{stronghatv} and (A11), 
we can verify that there exists a solution of \ref{P}. 
\qed  
\end{prth1.2e}

\vspace{10pt}
 
\section{Uniqueness for \ref{P}}\label{Sec5}

This section proves uniqueness of solutions to \ref{P}.     
\begin{prth1.2u}
We let $(\theta, \varphi)$, $(\overline{\theta}, \overline{\varphi})$ 
be two solutions of \ref{P} 
and put 
$\widetilde{\theta}:=\theta-\overline{\theta}$, 
$\widetilde{\varphi}:=\varphi-\overline{\varphi}$. 
Then the identity \eqref{df2}, the Young inequality, (A6), (A11), Lemma \ref{lemkuri1},  
the continuity of the embedding $V \hookrightarrow H$ and (A2) 
imply that 
there exists a constant $C_{1}=C_{1}(T)>0$ satisfying 
\begin{align}\label{cocoro1}
&\frac{1}{2}\frac{d}{dt}\left\|L^{1/2}\frac{d\widetilde{\varphi}}{dt}(t)\right\|_{H}^2 
+ \left(B_{1}\frac{d\widetilde{\varphi}}{dt}(t), \frac{d\widetilde{\varphi}}{dt}(t)\right)_{H} 
+ \frac{1}{2}\frac{d}{dt}\Bigl\|A_{2}^{1/2}\widetilde{\varphi}(t)\Bigr\|_{H}^2 
\notag \\[1mm] 
&= \left(B_{2}\widetilde{\theta}(t), \frac{d\widetilde{\varphi}}{dt}(t)\right)_{H} 
     - \left(\Phi\varphi(t)-\Phi\overline{\varphi}(t), 
                                                          \frac{d\widetilde{\varphi}}{dt}(t)\right)_{H} 
      - \left({\cal L}\varphi(t)-{\cal L}\overline{\varphi}(t), 
                                                          \frac{d\widetilde{\varphi}}{dt}(t)\right)_{H} 
\notag \\ 
&\leq \left(B_{2}\widetilde{\theta}(t), \frac{d\widetilde{\varphi}}{dt}(t)\right)_{H}  
         + \frac{C_{\Phi}^2}{2}(1+\|\varphi(t)\|_{V}^{p} + \|\overline{\varphi}(t)\|_{V}^{q})^2
                                                                               \|\widetilde{\varphi}(t)\|_{V}^2 
\notag \\ 
&\,\quad+ \frac{C_{{\cal L}}^2}{2}\|\widetilde{\varphi}(t)\|_{H}^2 
         + \left\|\frac{d\widetilde{\varphi}}{dt}(t)\right\|_{H}^2 
\notag \\[1.5mm]
&\leq  \left(B_{2}\widetilde{\theta}(t), \frac{d\widetilde{\varphi}}{dt}(t)\right)_{H} 
         + C_{1}\|\widetilde{\varphi}(t)\|_{V}^2 
         + \frac{1}{c_{L}}\left\|L^{1/2}\frac{d\widetilde{\varphi}}{dt}(t)\right\|_{H}^2    
\end{align}
for a.a.\ $t \in (0, T)$. 
Here we have from the Young inequality, (A2) and 
the continuity of the embedding $V \hookrightarrow H$ that 
there exists a constant $C_{2}>0$ such that 
\begin{align}\label{cocoro2}
\frac{1}{2}\frac{d}{dt}\|\widetilde{\varphi}(t)\|_{H}^2 
= \left(\frac{d\widetilde{\varphi}}{dt}(t), \widetilde{\varphi}(t)\right)_{H} 
\leq \frac{1}{2c_{L}}\left\|L^{1/2}\frac{d\widetilde{\varphi}}{dt}(t)\right\|_{H}^2 
       + C_{2}\|\widetilde{\varphi}(t)\|_{V}^2    
\end{align}
for a.a.\ $t \in (0, T)$. 
We see from (A3) that 
\begin{align}\label{cocoro3}
&\frac{1}{2}\Bigl\|A_{2}^{1/2}\widetilde{\varphi}(t)\Bigr\|_{H}^2 
+ \frac{1}{2}\|\widetilde{\varphi}\|_{H}^2 
\notag \\ 
&= \frac{1}{2}\langle A_{2}^{*}\widetilde{\varphi}(t), 
                                            \widetilde{\varphi}(t) \rangle_{V^{*}, V} 
    + \frac{1}{2}\|\widetilde{\varphi}\|_{H}^2 
\geq \frac{\omega_{2, 1}}{2}\|\widetilde{\varphi}(t)\|_{V}^2. 
\end{align}
Moreover, the identity \eqref{df1} yields that 
\begin{align}\label{cocoro4}
\left(B_{2}\widetilde{\theta}(t), \frac{d\widetilde{\varphi}}{dt}(t)\right)_{H} 
&= \frac{1}{\eta}\left(B_{2}\widetilde{\theta}(t), 
          - \frac{d\widetilde{\theta}}{dt}(t) - A_{1}\widetilde{\theta}(t) \right)_{H} 
\notag \\ 
&= -\frac{1}{2\eta}\frac{d}{dt}\Bigl\|B_{2}^{1/2}\widetilde{\theta}(t)\Bigr\|_{H}^2 
     - \frac{1}{\eta}(B_{2}\widetilde{\theta}(t), A_{1}\widetilde{\theta}(t))_{H}.  
\end{align}
Therefore it follows from \eqref{cocoro1}-\eqref{cocoro4} and (A4) that 
there exists a constant $C_{3}=C_{3}(T)>0$ such that 
\begin{align*}
&\frac{1}{2}\left\|L^{1/2}\frac{d\widetilde{\varphi}}{dt}(t)\right\|_{H}^2 
+ \frac{\omega_{2, 1}}{2}\|\widetilde{\varphi}(t)\|_{V}^2 
+ \frac{1}{2\eta}\|B_{2}^{1/2}\widetilde{\theta}(t)\|_{H}^2 
\notag \\ 
&\leq C_{3}\int_{0}^{t}\left\|L^{1/2}\frac{d\widetilde{\varphi}}{dt}(s)\right\|_{H}^2\,ds 
         + C_{3}\int_{0}^{t}\|\widetilde{\varphi}(s)\|_{V}^2\,ds      
\end{align*} 
for a.a.\ $t \in (0, T)$, 
whence we obtain that $\frac{d\widetilde{\varphi}}{dt}=\widetilde{\varphi}=0$ 
by the Gronwall lemma and (A2). 
Then the identity \eqref{df1} leads to the identity  
\begin{align}\label{cocoro5}
\frac{1}{2}\frac{d}{dt}\|\widetilde{\theta}(t)\|_{H}^2 
+ (A_{1}\widetilde{\theta}(t), \widetilde{\theta}(t))_{H} 
= 0.  
\end{align}
Thus it holds that $\widetilde{\theta}=0$.  
\qed
\end{prth1.2u}

\vspace{10pt}
 
\section{Error estimates}\label{Sec6}

In this section we will show Theorem \ref{erroresti}.   

\begin{prth1.3}
Let $h_{4}$ be as in Lemma \ref{lemkuri2}.  
Then, putting $z:=\frac{dv}{dt}$, 
we derive from the identity $\frac{d\widehat{v}_{h}}{dt}=\overline{z}_{h}$, 
the second equation in \ref{Ph} and \eqref{df2} that 
\begin{align}\label{e1}
&\frac{1}{2}\frac{d}{dt}\|L^{1/2}(\widehat{v}_{h}(t)-v(t))\|_{H}^2 
\notag \\ 
&= (L(\overline{z}_{h}(t)-z(t)), \widehat{v}_{h}(t)-\overline{v}_{h}(t))_{H} 
     + (L(\overline{z}_{h}(t)-z(t)), \overline{v}_{h}(t)-v(t))_{H} 
\notag \\ 
&= (L(\overline{z}_{h}(t)-z(t)), \widehat{v}_{h}(t)-\overline{v}_{h}(t))_{H} 
     - (B_{1}(\overline{v}_{h}(t)-v(t)), \overline{v}_{h}(t)-v(t))_{H} 
\notag \\ 
&\,\quad - (A_{2}(\overline{\varphi}_{h}(t)-\varphi(t)), \overline{v}_{h}(t)-v(t))_{H} 
- (\Phi\overline{\varphi}_{h}(t)-\Phi\varphi(t), \overline{v}_{h}(t)-v(t))_{H} 
\notag \\ 
&\,\quad- ({\cal L}\overline{\varphi}_{h}(t)-{\cal L}\varphi(t), 
                                                                          \overline{v}_{h}(t)-v(t))_{H} 
+ (B_{2}(\overline{\theta}_{h}(t) - \theta(t)), \overline{v}_{h}(t)-v(t))_{H}.    
\end{align}
Here the boundedness of the operator $L : H \to H$ 
implies that 
there exists a constant $C_{1}>0$ such that 
\begin{align}\label{e2}
(L(\overline{z}_{h}(t)-z(t)), \widehat{v}_{h}(t)-\overline{v}_{h}(t))_{H} 
&\leq \|L(\overline{z}_{h}(t)-z(t))\|_{H}\|\widehat{v}_{h}(t)-\overline{v}_{h}(t)\|_{H} 
\notag \\ 
&\leq C_{1}\|\overline{z}_{h}(t)-z(t)\|_{H}\|\widehat{v}_{h}(t)-\overline{v}_{h}(t)\|_{H} 
\end{align}
for a.a.\ $t \in (0, T)$ and all $h \in (0, h_{4})$. 
We have from 
the identities $\overline{v}_{h}=\frac{d\widehat{\varphi}_{h}}{dt}$, $v=\frac{d\varphi}{dt}$ 
and the boundedness of the operator $A_{2}^* : V \to V^*$ that 
there exists a constant $C_{2}>0$ such that  
\begin{align}\label{e3}
&- (A_{2}(\overline{\varphi}_{h}(t)-\varphi(t)), \overline{v}_{h}(t)-v(t))_{H} 
\notag \\ 
&= - \langle A_{2}^{*}(\overline{\varphi}_{h}(t)-\widehat{\varphi}_{h}(t)), 
                                                    \overline{v}_{h}(t)-v(t) \rangle_{V^{*}, V} 
- \frac{1}{2}\frac{d}{dt}\|A_{2}^{1/2}(\widehat{\varphi}_{h}(t)-\varphi(t))\|_{H}^2 
\notag \\ 
&\leq C_{2}\|\overline{\varphi}_{h}(t)-\widehat{\varphi}_{h}(t)\|_{V}
                                                                   \|\overline{v}_{h}(t)-v(t)\|_{V} 
- \frac{1}{2}\frac{d}{dt}\|A_{2}^{1/2}(\widehat{\varphi}_{h}(t)-\varphi(t))\|_{H}^2 
\end{align}
for a.a.\ $t \in (0, T)$ and all $h \in (0, h_{4})$. 
We see from (A6), Lemma \ref{lemkuri1}, the Young inequality and (A2) that 
there exists a constant $C_{3}=C_{3}(T)>0$ such that 
\begin{align}\label{e4}
&- (\Phi\overline{\varphi}_{h}(t)-\Phi\varphi(t), \overline{v}_{h}(t)-v(t))_{H} 
\notag \\ 
&\leq C_{\Phi}(1 + \|\overline{\varphi}_{h}(t)\|_{V}^{p} + \|\varphi(t)\|_{V}^{q})
               \|\overline{\varphi}_{h}(t)-\varphi(t)\|_{V}\|\overline{v}_{h}(t)-v(t)\|_{H} 
\notag \\ 
&\leq C_{3}\|\overline{\varphi}_{h}(t)-\varphi(t)\|_{V}\|\overline{v}_{h}(t)-v(t)\|_{H} 
\notag \\ 
&\leq \frac{C_{3}}{2}\|\overline{\varphi}_{h}(t)-\varphi(t)\|_{V}^2 
         + \frac{C_{3}}{2}\|\overline{v}_{h}(t)-v(t)\|_{H}^2 
\notag \\ 
&\leq C_{3}\|\overline{\varphi}_{h}(t)-\widehat{\varphi}_{h}(t)\|_{V}^2 
         + C_{3}\|\widehat{\varphi}_{h}(t)-\varphi(t)\|_{V}^2 
\notag \\
    &\,\quad+ C_{3}\|\overline{v}_{h}(t)-\widehat{v}_{h}(t)\|_{H}^2 
         + \frac{C_{3}}{c_{L}}\|L^{1/2}(\widehat{v}_{h}(t)-v(t))\|_{H}^2 
\end{align}
for a.a.\ $t \in (0, T)$ and all $h \in (0, h_{4})$. 
The condition (A11), the continuity of the embedding $V \hookrightarrow H$, 
the Young inequality and (A2) mean that 
there exists a constant $C_{4}>0$ such that 
\begin{align}\label{e5}
&- ({\cal L}\overline{\varphi}_{h}(t)-{\cal L}\varphi(t), \overline{v}_{h}(t)-v(t))_{H} 
\notag \\ 
&\leq C_{4}\|\overline{\varphi}_{h}(t)-\varphi(t)\|_{V}\|\overline{v}_{h}(t)-v(t)\|_{H} 
\notag \\ 
&\leq \frac{C_{4}}{2}\|\overline{\varphi}_{h}(t)-\varphi(t)\|_{V}^2 
         + \frac{C_{4}}{2}\|\overline{v}_{h}(t)-v(t)\|_{H}^2 
\notag \\ 
&\leq C_{4}\|\overline{\varphi}_{h}(t)-\widehat{\varphi}_{h}(t)\|_{V}^2 
         + C_{4}\|\widehat{\varphi}_{h}(t)-\varphi(t)\|_{V}^2 
\notag \\
    &\,\quad+ C_{4}\|\overline{v}_{h}(t)-\widehat{v}_{h}(t)\|_{H}^2 
         + \frac{C_{4}}{c_{L}}\|L^{1/2}(\widehat{v}_{h}(t)-v(t))\|_{H}^2   
\end{align}
for a.a.\ $t \in (0, T)$ and all $h \in (0, h_{4})$. 
It follows from the first equation in \ref{Ph} and \eqref{df1} that 
\begin{align}\label{e6}
&(B_{2}(\overline{\theta}_{h}(t) - \theta(t)), \overline{v}_{h}(t)-v(t))_{H} 
\notag \\ 
&= - \frac{1}{\eta}\Bigl(B_{2}(\overline{\theta}_{h}(t) - \theta(t)), 
             \frac{d\widehat{\theta}_{h}}{dt}(t)-\frac{d\theta}{dt}(t)\Bigr)_{H} 
\notag \\
&\,\quad- \frac{1}{\eta}(B_{2}(\overline{\theta}_{h}(t) - \theta(t)), 
                                                 A_{1}(\overline{\theta}_{h}(t) - \theta(t)))_{H} 
\notag \\ 
&= - \frac{1}{\eta}
           \Bigl\langle B_{2}^{*}(\overline{\theta}_{h}(t) - \widehat{\theta}_{h}(t)), 
                  \frac{d\widehat{\theta}_{h}}{dt}(t)-\frac{d\theta}{dt}(t) \Bigr\rangle_{V^*, V} 
     - \frac{1}{2\eta}\frac{d}{dt}\|B_{2}^{1/2}(\widehat{\theta}_{h}(t) - \theta(t))\|_{H}^2  
\notag \\ 
&\,\quad - \frac{1}{\eta}(B_{2}(\overline{\theta}_{h}(t) - \theta(t)), 
                                                 A_{1}(\overline{\theta}_{h}(t) - \theta(t)))_{H}.  
\end{align}
Hence we derive from \eqref{e1}-\eqref{e6}, 
the integration over $(0, t)$, where $t \in [0, T]$, 
the boundedness of the operator $B_{2}^{*} : V \to V^*$, 
\eqref{rem4}-\eqref{rem6}, 
Lemmas \ref{lemkuri4} and \ref{lemkuri6} that 
there exists a constant $C_{5}=C_{5}(T)>0$ such that 
\begin{align}\label{e7}
&\frac{1}{2}\|L^{1/2}(\widehat{v}_{h}(t)-v(t))\|_{H}^2 
+ \frac{1}{2}\|A_{2}^{1/2}(\widehat{\varphi}_{h}(t)-\varphi(t))\|_{H}^2 
+ \int_{0}^{t}\|B_{1}^{1/2}(\overline{v}_{h}(s)-v(s))\|_{H}^2\,ds  
\notag \\ 
&+ \frac{1}{2\eta}\|B_{2}^{1/2}(\widehat{\theta}_{h}(t) - \theta(t))\|_{H}^2  
  + \frac{1}{\eta}\int_{0}^{t}(B_{2}(\overline{\theta}_{h}(s) - \theta(s)), 
                                                 A_{1}(\overline{\theta}_{h}(s) - \theta(s)))_{H}\,ds   
\notag \\ 
&\leq C_{5}h   
+ C_{5}\int_{0}^{t}\|\widehat{\varphi}_{h}(s)-\varphi(s)\|_{V}^2\,ds  
+ C_{5}\int_{0}^{t}\|L^{1/2}(\widehat{v}_{h}(s)-v(s))\|_{H}^2\,ds  
\end{align}
for all $t \in [0, T]$ and all $h \in (0, h_{4})$. 
Here 
the identities $\frac{d\widehat{\varphi}_{h}}{dt}=\overline{v}_{h}$, $\frac{d\varphi}{dt}=v$, 
the Young inequality, (A2) 
and the continuity of the embedding $V \hookrightarrow H$ yield that 
there exists a constant $C_{6}>0$ such that 
\begin{align}\label{e8}
&\frac{1}{2}\frac{d}{dt}\|\widehat{\varphi}_{h}(t)-\varphi(t)\|_{H}^2 
\notag \\ 
&= (\overline{v}_{h}(t)-v(t), \widehat{\varphi}_{h}(t)-\varphi(t))_{H} 
\notag \\ 
&\leq \frac{1}{2}\|\overline{v}_{h}(t)-v(t)\|_{H}^2 
         + \frac{1}{2}\|\widehat{\varphi}_{h}(t)-\varphi(t)\|_{H}^2 
\notag \\ 
&\leq \|\overline{v}_{h}(t)-\widehat{v}_{h}(t)\|_{H}^2 
         + \frac{1}{c_{L}}\|L^{1/2}(\widehat{v}_{h}(t) - v(t))\|_{H}^2  
         + C_{6}\|\widehat{\varphi}_{h}(t)-\varphi(t)\|_{V}^2 
\end{align}
for a.a.\ $t \in (0, T)$ and all $h \in (0, h_{4})$. 
Thus, integrating \eqref{e8} over $(0, t)$, where $t \in [0, T]$, 
we deduce from \eqref{e7} and (A3) 
that there exists a constant $C_{7}=C_{7}(T)>0$ such that 
\begin{align}\label{e9}
&\frac{1}{2}\|L^{1/2}(\widehat{v}_{h}(t)-v(t))\|_{H}^2 
+ \frac{\omega_{2, 1}}{2}\|\widehat{\varphi}_{h}(t) - \varphi(t)\|_{V}^2 
+ \int_{0}^{t}\|B_{1}^{1/2}(\overline{v}_{h}(s)-v(s))\|_{H}^2\,ds  
\notag \\ 
&+ \frac{1}{2\eta}\|B_{2}^{1/2}(\widehat{\theta}_{h}(t) - \theta(t))\|_{H}^2  
  + \frac{1}{\eta}\int_{0}^{t}(B_{2}(\overline{\theta}_{h}(s) - \theta(s)), 
                                                 A_{1}(\overline{\theta}_{h}(s) - \theta(s)))_{H}\,ds 
\notag \\ 
&\leq C_{7}h   
+ C_{7}\int_{0}^{t}\|\widehat{\varphi}_{h}(s)-\varphi(s)\|_{V}^2\,ds  
+ C_{7}\int_{0}^{t}\|L^{1/2}(\widehat{v}_{h}(s)-v(s))\|_{H}^2\,ds  
\end{align}    
for all $t \in [0, T]$ and all $h \in (0, h_{4})$. 

Next the first equation in \ref{Ph} and \eqref{df1} lead to the identity
\begin{align}\label{e10}
&\frac{1}{2}\frac{d}{dt}\|\widehat{\theta}_{h}(t) - \theta(t)\|_{H}^2 
\notag \\ 
&= - \eta(\overline{v}_{h}(t)-v(t), \widehat{\theta}_{h}(t)-\theta(t))_{H} 
    - (A_{1}(\overline{\theta}_{h}(t)-\theta(t)), 
                              \widehat{\theta}_{h}(t) - \overline{\theta}_{h}(t))_{H} 
\notag \\
    &\,\quad- \langle A_{1}^{*}(\overline{\theta}_{h}(t)-\theta(t)), 
                                 \overline{\theta}_{h}(t)-\theta(t) \rangle_{V^{*}, V}.  
\end{align}
Here we use the Young inequality and (A2) to infer that 
\begin{align}\label{e11}
&- (\overline{v}_{h}(t)-v(t), \widehat{\theta}_{h}(t)-\theta(t))_{H} 
\notag \\ 
&\leq \frac{1}{2}\|\overline{v}_{h}(t)-v(t)\|_{H}^2 
       + \frac{1}{2}\|\widehat{\theta}_{h}(t)-\theta(t)\|_{H}^2 
\notag \\ 
&\leq \|\overline{v}_{h}(t)-\widehat{v}_{h}(t)\|_{H}^2 
         + \|\widehat{v}_{h}(t)-v(t)\|_{H}^2 
         + \frac{1}{2}\|\widehat{\theta}_{h}(t)-\theta(t)\|_{H}^2 
\notag \\ 
&\leq \|\overline{v}_{h}(t)-\widehat{v}_{h}(t)\|_{H}^2 
         + \frac{1}{c_{L}}\|L^{1/2}(\widehat{v}_{h}(t)-v(t))\|_{H}^2 
         + \frac{1}{2}\|\widehat{\theta}_{h}(t)-\theta(t)\|_{H}^2.  
\end{align} 
We have from (A3) that 
\begin{align}\label{e12}
&- \langle A_{1}^{*}(\overline{\theta}_{h}(t)-\theta(t)), 
                                      \overline{\theta}_{h}(t)-\theta(t) \rangle_{V^{*}, V} 
\notag \\ 
&\leq -\omega_{1, 1}\|\overline{\theta}_{h}(t)-\theta(t)\|_{V}^2 
         + \|\overline{\theta}_{h}(t)-\theta(t)\|_{H}^2  
\notag \\ 
&\leq -\omega_{1, 1}\|\overline{\theta}_{h}(t)-\theta(t)\|_{V}^2 
         + 2\|\overline{\theta}_{h}(t) - \widehat{\theta}_{h}(t)\|_{H}^2 
         + 2\|\widehat{\theta}_{h}(t) - \theta(t)\|_{H}^2.  
\end{align} 
Hence, owing to \eqref{e10}-\eqref{e12}, 
the integration over $(0, t)$, where $t \in [0, T]$, 
\eqref{rem5}, \eqref{rem6}, 
Lemmas \ref{lemkuri4} and \ref{lemkuri6}, 
there exists a constant $C_{8}=C_{8}(T)>0$ such that 
\begin{align}\label{e13}
&\frac{1}{2}\|\widehat{\theta}_{h}(t) - \theta(t)\|_{H}^2 
+ \omega_{1, 1}\int_{0}^{t}\|\overline{\theta}_{h}(s) - \theta(s)\|_{V}^2\,ds 
\notag \\ 
&\leq C_{8}h 
+ C_{8}\int_{0}^{t} \|L^{1/2}(\widehat{v}_{h}(s)-v(s))\|_{H}^2\,ds 
+ C_{8}\int_{0}^{t}\|\widehat{\theta}_{h}(s)-\theta(s)\|_{H}^2\,ds 
\end{align}
for all $t \in [0, T]$ and all $h \in (0, h_{4})$. 
 
Therefore we can obtain Theorem \ref{erroresti} 
by combining \eqref{e9}, \eqref{e13} 
and by applying the Gronwall lemma. 
\qed 
\end{prth1.3}

\section*{Acknowledgments}
The author is supported by JSPS Research Fellowships 
for Young Scientists (No.\ 18J21006).   
%
%
%

\end{document}